\documentclass{article}
\usepackage{tikz}
\usepackage{amsmath,amsfonts,amssymb,amsthm}
\usepackage{graphicx}
\usepackage{kotex}
\usepackage{color}
\usepackage{tikz}
\usetikzlibrary{math}

\newcommand{\bR}{\mathbb{R}}

\newcommand{\mr}[1]{\mathrm{#1}}
\newcommand{\mb}[1]{\mathbf{#1}}

\theoremstyle{definition}
\newtheorem{defn}{Definition}[section]
\newtheorem{lem}[defn]{Lemma}
\newtheorem{prop}[defn]{Proposition}
\newtheorem{thm}[defn]{Theorem}
\newtheorem*{thm*}{Theorem}
\newtheorem{cor}[defn]{Corollary}
\theoremstyle{remark}
\newtheorem{ex}[defn]{Example}

\numberwithin{equation}{section}

\title{Convergence of magnitude of finite positive definite metric spaces}
\author{Byungchang So\\ \small
Department of Mathematical Sciences, \\ \small
Seoul National University \\ \small %
}
\date{}
\begin{document}
\maketitle

\begin{abstract}
The magnitude of metric spaces does not appear to possess a simple, convenient continuity property, and previous studies have presented affirmative results under additional constraints or weaker notions, as well as counterexamples. 
In this vein, we discuss the continuity of magnitude of finite positive definite metric spaces with respect to the Gromov-Hausdorff distance, but with a restriction of the domain based on a canonical partition of a sufficiently small neighborhood of a finite metric space. 
As a result, the main theorem of this article explains a condition on the cardinality of metric spaces that determines the continuity of magnitude. 
This study takes advantage of the geometric interpretation of magnitude as the circumradius of the corresponding finite Euclidean subset. 
Such a transformation is especially useful for constructing counterexamples, as we can depend on Euclidean geometric intuition.
\end{abstract}

\section{Introduction}

The magnitude \cite{Leinster2013} of a compact metric space is a real number representing a certain notion of size. 
Being related to other concepts like cardinality, dimension \cite{Meckes2015}, volume \cite{BarceloCarbery2018}, and so on, magnitude also has peculiarities represented by the phrase ``the effective number of points.'' 
Distinctive properties of magnitude also appear, e.g., in its dependency upon scaling, as propositions in \cite{Meckes2015,GimperleinGoffeng2021}, examples in \cite{LeinsterWillerton2013}, and applications to data analysis \cite{Bunch+2020,Limbeck+2024}. 
Among several directions for understanding magnitude is the study of its continuity.

As \cite{Leinster2013} shows, magnitude is not continuous in general with respect to the Gromov-Hausdorff distance. 
\cite{Meckes2013} provided partial affirmative answers, namely that magnitude is lower semicontinuous when restricted to positive definite metric spaces and that magnitude is continuous when restricted to positively weighted metric spaces. 
Even under some relatively strong restrictions, the continuity is not regained immediately: 
for subsets of Euclidean spaces and with respect to Hausdorff distance in $\bR^d$, magnitude is continuous when restricted to convex subsets \cite{LeinsterMeckes2016}, but not in general \cite{Gimperlein+2022}. 
\cite{Katsumasa+2025} shows that magnitude is nowhere continuous among finite metric spaces and that a slightly weaker notion, ``generic'' continuity, holds. 
As we do not have one convenient continuity theorem, it is natural to refine the discussion and try various additional assumptions, weaker notions, and so on. 

In this vein, this article examines the continuity of magnitude of finite positive definite metric spaces with respect to the Gromov-Hausdorff distance, but with a specific restriction of domain of magnitude. 
More precisely, we can observe that any ball in Gromov-Hausdorff distance with center $X$ and sufficiently small radius can be partitioned in a canonical way. 
Roughly speaking, we classify any other metric space $X'$ close enough to $X$ into countably many types, according to the count of elements in $X'$ that can be considered close to $x$ for each $x \in X$. 
In a later part of this article, we precisely define the term ``clustered in type $r$,'' and the concept turns out critical in the convergence of magnitude as our main theorem states:
\begin{thm} \label{thm:main}
Let $X$ and $X_n$ $(n=1,2,\cdots)$ be finite positive definite metric spaces, $r = \langle r_1,r_2,\cdots,r_m \rangle$ a finite monotone decreasing sequence of natural numbers, $\|r \|_1 := r_1 + r_2 + \cdots$ the sum of $r$, and $k$ a natural number. 
Then the statement
\begin{quote}
If $\lim_{n \rightarrow \infty} X_n = X$ in Gromov-Hausdorff distance, each $X_n$ $(n=1,2,\cdots)$ is clustered in type $r$, and $\#X = k$, then $\lim_{n\rightarrow \infty} |X_n| = |X|$.
\end{quote}
is true if and only if $\|r\|_1 \leq 2$, $k = 1,2,\cdots$ or $\| r\|_1=3$, $k=1$.
\end{thm}
And we immediately obtain the following corollary, which explains how the cardinality matter for the convergence of magnitude:
\begin{cor}
Let $X$ and $X_n$ $(n=1,2,\cdots)$ be finite positive definite metric spaces. 
Then the statement
\begin{equation*}
\lim_{n \rightarrow \infty} X_n = X,\ \#X_n \leq k \quad \Rightarrow \quad \lim_{n \rightarrow \infty} |X_n| = |X|
\end{equation*}
is true if and only if $k\leq \#X+2$ or $\#X=1, k=4$.
\end{cor}

Meanwhile, the restriction to finite positive definite metric spaces enables the geometric approach (this was proposed by \cite{AsaoGomi2025,Devriendt2025} while the previous version of the current article was being written), transforming metric spaces into finite Euclidean subsets and associating magnitudes to circumradii. 
The transform is especially useful in constructing counterexamples because we can depend on Euclidean geometry instead of abstractness of metric spaces.

The remaining part of this article is organized as follows. 
In Section \ref{sec:mag-geom}, we review definitions of relevant concepts and the geometric approach recently suggested. 
Section \ref{sec:type-wise} contains the main result of this article, the refined observation on the continuity of magnitude. 
All the routine computations in Section \ref{sec:type-wise} are deferred to Appendix \ref{sec:computation}.

\subsection{List of Symbols}
Below is a summary of the notation used throughout this article.
\begin{itemize}
\item $\#X$ : the cardinality of a finite set $X$ 
\item $|X|$ : the magnitude of a metric space $X$
\item $d_{GH}(X,X')$ : the Gromov-Hausdorff distance between two metric spaces $X$ and $X'$
\item $d_{H}(Y,Y')$ : the Hausdorff distance between two subsets $Y$ and $Y'$ of an ambient space
\item $\| y \|$ : the Euclidean norm of $y$ in a Euclidean space
\item $\rho_Y$: the circumradius of a Euclidean subset $Y$ in general position
\item $\mathrm{K}_Y$: the circumcenter of a Euclidean subset $Y$ in general position
\end{itemize}

\section{Geometric Interpretation of Magnitude} \label{sec:mag-geom}

We briefly review the magnitude of metric spaces in a form tailored to our purpose, referring readers to \cite{LeinsterMeckes2016} for general properties and examples.
Let $X=\{x_1,\cdots,x_k \}$ be a finite metric space with distance function $d:X\times X \rightarrow \bR_{\geq 0}$. 
The \textbf{zeta matrix} (or \textbf{similarity matrix}) of $X$ is $\zeta_X = (e^{-d(x_i,x_j)})_{1 \leq i,j \leq k}$. 
The metric space $X$ is \textbf{positive definite} \cite{Meckes2013} if $\zeta_X$ is a positive definite matrix. 
We abbreviate ``finite positive definite metric space'' as FPDMS. 
The \textbf{magnitude} $|X|$ of an FPDMS $X$ is the sum of entries of $\zeta_X^{-1}$, i.e.,
\begin{equation*} \label{eq:magn-linalg}
|X| = \frac{\text{cofactor-sum} (\zeta_X)}{\det \zeta_X} = \frac{\sum_{i,j=1}^k (-1)^{i+j} \det (\zeta_X^{(i,j)})}{\det \zeta_X},
\end{equation*}
where we denote by $\zeta_X^{(i,j)}$ the matrix obtained by removing the $i-$th row and the $j-$th column of a matrix $\zeta_X$. 

The magnitude of an FPDMS can be interpreted geometrically (see \cite{AsaoGomi2025,Devriendt2025} also) as follows. 
Given an FPDMS $X = \{x_1 ,\cdots, x_k \}$, due to positive definiteness, the zeta matrix $\zeta_X$ is equal to the Gram matrix of some subset $Y = \{y_1,\cdots,y_k\} \subset \bR^{k}$:
\begin{equation} \label{eq:dot-vs-sim}
y_i \cdot y_j = e^{-d(x_i,x_j)} \qquad (i,j=1,2,\cdots,k ).
\end{equation}
From volume formulas
\begin{align*} 
\det \zeta_X 
&=  \big[ \mathrm{Vol} (y_1,y_2,\cdots ,y_k)\big ] ^2  \\
\text{cofactor-sum}(\zeta_X) 
&= \big[\mathrm{Vol} (y_2 -y_1, y_3 -y_1,\cdots, y_k -y_1) \big]^2
\end{align*}
in elementary geometric linear algebra, where $\mathrm{Vol}(\cdots )$ refers to the signed volume of the parallelopiped generated by finitely many given vectors, we immediately have
\begin{equation} \label{eq:mag-as-vol}
|X| = \left[ \frac{\mathrm{Vol} (y_2 -y_1, y_3 -y_1,\cdots, y_k -y_1)}{\mathrm{Vol} (y_1,y_2,\cdots ,y_k)} \right]^2 .
\end{equation}

Let $h$ be the distance from the origin to the affine subspace $\mathrm{aff}\ Y$ spanned by $Y$. 
By calculating the volume of simplex $\{\sum_{i=1}^k t_i y_i \in \bR^k: t_i \geq 0,~ \sum_{i=1}^k t_i = 1 \}$ in two different ways, we obtain
\begin{equation} \label{eq:volume-2ways}
\frac{\mathrm{Vol} (y_1,y_2,\cdots ,y_k)}{k!}  = \frac{1}{k} \times h \times \frac{\mathrm{Vol} (y_2 -y_1, y_3 -y_1,\cdots, y_k -y_1)}{(k-1)!},
\end{equation}
where the right-hand side is obtained by regarding $\mathrm{aff}\ Y$ as base and $h$ as height. 
Meanwhile, $Y$ is contained in the unit sphere $S^{k-1}=\{y\in \bR^{k}: \|y\| =1\}$, and $S^{k-1} \cap \mathrm{aff}\ Y$ is equal to the circumsphere of $Y$. 
This implies the equality
\begin{equation} \label{eq:h-vs-radius}
h = \sqrt{1 - \rho_Y ^2}.
\end{equation}
Combining equations \eqref{eq:mag-as-vol}, \eqref{eq:volume-2ways}, and \eqref{eq:h-vs-radius} we conclude $|X| = \frac{1}{1-\rho_Y^2}$.
Following \cite{Devriendt2025}, we refer to $Y$ as a \textbf{similarity embedding} of $X$, or just a similarity embedding without explicitly mentioning the original metric space $X$ (See Figure \ref{fig:embed} for an illustration). 

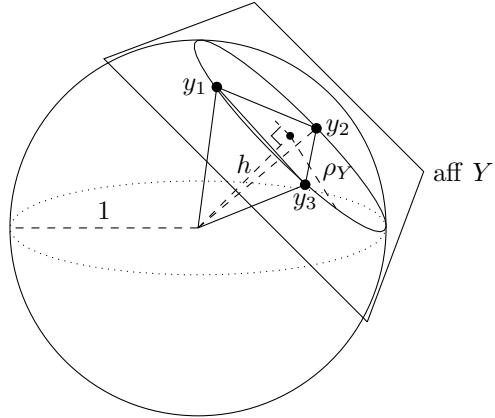
\begin{figure}
\centering
\begin{tikzpicture}
\tikzmath{\radius = 2.5; \radiusb = 0.25*\radius;
\yax =\radius*0.1; \yay = \radius*0.75;
\ybx = \radius*0.63; \yby = \radius*0.53;
\ycx = \radius*0.57; \ycy = \radius*0.23;
\kpos = 0.49; \krada = \radius*0.71; \kradb = \radius*0.13;
}
\draw (\radius,0) arc(0:360:\radius);
\draw[dashed] (0,0) -- (-\radius,0);
\draw (-0.5*\radius,0) node[anchor=south]{$1$};
\draw[dotted,domain=0:360,samples=40] plot ({\radius*cos(\x)},{\radiusb*sin(\x)});
\draw (0,0)--(\yax, \yay) node[fill,circle,minimum size=4pt,inner sep=0pt]{};
\draw[dashed] (0,0)--(\ybx, \yby) node[fill,circle,minimum size=4pt,inner sep=0pt]{};
\draw (0,0)--(\ycx, \ycy) node[fill,circle,minimum size=4pt,inner sep=0pt]{};
\draw (\yax,\yay) node[anchor=east]{$y_1$}
-- (\ybx,\yby) node[anchor=west]{$y_2$}
-- (\ycx,\ycy) node[anchor=north]{$y_3$}
-- (\yax,\yay);
\draw[dashed] (0,0) -- (\radius*\kpos,\radius*\kpos) node[fill,circle,minimum size=3pt,inner sep=0pt]{};
\draw (0.5*\radius*\kpos,0.5*\radius*\kpos) node[anchor=south]{$h$};
\draw ({\radius*(\kpos-0.05)},{\radius*(\kpos-0.05)}) 
-- ({\radius*(\kpos-0.10)},{\radius*(\kpos)}) 
-- ({\radius*(\kpos-0.05)},{\radius*(\kpos+0.05)});
\draw[dashed] (\radius*\kpos,\radius*\kpos) 
-- ({\radius*(\kpos-0.08)},{\radius*(\kpos+0.08)});
\draw[domain=0:360,samples=40] plot ({\radius*\kpos+\krada*cos(\x)/1.414+\kradb*sin(\x)/1.414},{\radius*\kpos-\krada*cos(\x)/1.414+\kradb*sin(\x)/1.414});
\draw[dashed] (\radius*\kpos,\radius*\kpos)
-- ({\radius*\kpos+\krada*cos(309)/1.414+\kradb*sin(309)/1.414},{\radius*\kpos-\krada*cos(309)/1.414+\kradb*sin(309)/1.414});
\draw ({\radius*\kpos+0.5*\krada*cos(309)/1.414+0.5*\kradb*sin(309)/1.414},{\radius*\kpos-0.5*\krada*cos(309)/1.414+0.5*\kradb*sin(309)/1.414}) node[anchor=west]{$\rho_Y$};

\draw (\radius*0.9,{\radius*(-0.5)}) 
-- (\radius*1.2, \radius*0.3) node[anchor=west]{$\mathrm{aff\ } Y$}
-- (\radius*0.3, \radius*1.2)
-- ({\radius*(-0.5)},\radius*0.9)
-- (\radius*0.9,{\radius*(-0.5)});
\end{tikzpicture}
\caption{Similarity embedding $Y=\{y_1,y_2,y_3\}$ of a $3-$point positive definite metric space.}
\label{fig:embed}
\end{figure}

There exist certain restrictions on similarity embeddings $Y$ of $X$.
First, $| \mathrm{Vol}(y_1, y_2, \cdots,y_k) |= \sqrt{\det \zeta_X} > 0 $ gives $h>0$, which in turn gives $\rho_Y =\sqrt{1-h^2} < 1$. 
Next, the relation 
\begin{equation} \label{eq:dist-x-y}
d(x_i,x_j) = \log \left( 1- \frac{1}{2}\|y_i -y_j \|^2 \right)^{-1}    
\end{equation}
following from \eqref{eq:dot-vs-sim} translates the triangle inequality of $d$ to
\begin{equation*}
\log \left( 1 - \frac{1}{2} \|y_i -y_j \|^2 \right)^{-1} + \log \left( 1 - \frac{1}{2} \|y_j -y_k \| \right)^{-1} \geq \log \left( 1 - \frac{1}{2} \|y_i -y_k \|^2 \right)^{-1}.
\end{equation*}
We will use an equivalent form
\begin{equation} \label{eq:strongly-acute}
\|y_i -y_j \|^2 + \|y_j -y_k \|^2 
\geq \|y_i -y_k \|^2 + \frac{1}{2} \|y_i -y_j \|^2 \|y_j -y_k \|^2 ,
\end{equation}
which we will call the \textbf{tri-similarity inequality}. 
Then, any similarity embedding $Y$ must at least satisfy $\rho_Y<1$ and the tri-similarity inequality \eqref{eq:strongly-acute} for each triple.

Conversely, let $Y = \{y_1,\cdots,y_k\} \subset \bR^{k-1}$ have a circumradius $\rho_Y<1$ and every triple in $Y$ satisfies tri-similarity inequality. 
From \eqref{eq:strongly-acute}, a $k-$point set $X'=\{x_1',\cdots, x_k'\}$ equipped with the distance function $d'(x_i',x_j') =1-\frac{1}{2} \|y_i -y_j\|^2$ becomes a metric space. 
Besides, $\rho_Y < 1$ implies the existence of a hypersphere with radius $1$ in $\bR^{k}$ containing $Y$. 
If $\rm K$ denotes the center of the sphere, then the equation
\begin{align*}
e^{-d(x_i', x_j')} = 1 - \frac{1}{2} \| y_i- y_j \|^2 = 1 - \frac{1}{2} \| \overrightarrow{\mr Ky_i} -\overrightarrow{\mr Ky_j} \|^2 = \overrightarrow{\mr Ky_i} \cdot \overrightarrow{\mr Ky_j}
\end{align*}
indicates that the zeta matrix $Z_{X'}$ of $X'$ is positive definite. 
Thus, $Y$ is a similarity embedding of FPDMS $X'$. 

The discussion above summarizes to the following.

\begin{thm}[{See \cite[Theorem 2.15]{Devriendt2025}, \cite[Theorem 1.1]{AsaoGomi2025} also}]
A subset $Y \subset \bR^d$ is a similarity embedding of an FPDMS $X$ if and only if $\rho_Y<1$ and every triple in $Y$ satisfies the tri-similarity inequality. 
And in such case we have
\begin{equation*}
|X| = \frac{1}{1-\rho_Y^2}
\end{equation*}
\end{thm}
To apply the geometric interpretation to the proof of Theorem \ref{thm:main}, we need to observe that the similarity embedding is bi-continuous in the sense of the following lemma.

\begin{lem}
Let $X$ be an FPDMS such that $\#X \leq d+1$ and $Y \subset \bR^d$ its similarity embedding. 
For $Y_1, Y_2 \subset \bR^d$, let us define $d_{H,\mathrm{rigid}}(Y_1,Y_2)$ as
\begin{equation*}
d_{H,\text{rigid}}(Y_1,Y_2) = \inf \big\{ d_H \big(Y_1, T(Y_2) \big): T \text{ is a rigid motion in } \bR^d \big\}.
\end{equation*}
\begin{enumerate}
\item Given $\varepsilon > 0$, there exists $\delta_d > 0$ depending on $d$ such that
\begin{equation*}
\#X' \leq d+1,\ d_{GH}(X,X') < \delta_d \quad \Rightarrow \quad d_{H,\textrm{rigid}} (Y, Y') < \varepsilon
\end{equation*}
holds for FPDMS $X'$ and its similarity embedding $Y' \subset \bR^d$.
\item Given $\varepsilon > 0$, there exists $\delta > 0$ such that 
\begin{equation*}
\#X' \leq d+1,\ d_{H,\textrm{rigid}}(Y,Y') < \delta \quad \Rightarrow \quad d_{GH} (X,X') < \varepsilon.
\end{equation*}
holds for FPDMS $X'$ and its similarity embedding $Y' \subset \bR^d$.
\end{enumerate}
\end{lem}

\begin{proof}
We use two auxiliary facts on Gromov-Hausdorff.
First, the Gromov-Hausdorff distance between two compact metric spaces $Z$ and $Z'$ can be represented \cite[Theorem 2.1]{KaltonOstrovskii1999} (see \cite{Memoli2013} also) as
\begin{equation} \label{eq:GH-distort}
\begin{aligned}
d_{GH} (Z, Z') \qquad\qquad\qquad&\\
= \frac{1}{2} \inf_{\substack{f: Z \rightarrow Z' \\ f' : Z' \rightarrow Z}} 
\max \Big\{ \sup_{z_1, z_2 \in Z} 
& \big \{ \big|d_Z(z_1, z_2) - d_{Z'} \big( f(z_1 ), f(z_2) \big) \big| \big \},   \\
\sup _{z_1', z_2' \in Z'}  
&\big\{ \big| d_Z \big(f'(z_1' ), f'(z_2' ) \big) - d_{Z'} (z_1' ,z_2' ) \big| \big\}, \\
 \sup_{z_1\in Z, z_2' \in Z'} 
&\big\{ \big| d_Z \big(z_1, f'(z_2')\big) - d_{Z'} \big(f(z_1), z_2'\big)\big| \big \}   \ \ 
\Big\}.
\end{aligned}
\end{equation}
Second, we can relate \cite[Theorem 2]{Memoli2008} $d_{GH}(Y,Y')$ and Hausdorff distance in $\bR^d$ as
\begin{equation*}
d_{GH}(Y,Y') \leq  d_{H,\text{rigid}}(Y,Y') \leq c_d  \cdot \max\{\mathrm{diam}\ Y,\mathrm{diam}\ Y' \} \cdot d_{GH}(Y,Y'),
\end{equation*}
where $c_d$ is a constant which depends only on the dimension $d$. 

Let bijections $h:X \rightarrow Y$ and $h': X' \rightarrow Y'$ be such that
\begin{equation*}
d_X (x_1, x_2) = \varphi \Big(d_Y \big(h(x_1), h(x_2)\big)\Big),
\qquad 
d_{X'} (x_1', x_2') = \varphi \Big(d_{Y'} \big(h'(x_1'), h'(x_2')\big)\Big),
\end{equation*}
where $\varphi(t) = -\log \left( 1 - \frac{t^2}{2} \right)$. 
Suppose that $\epsilon > 0$ is given.
\begin{enumerate}
\item By the uniform continuity of $\varphi^{-1}$ on interval $[0,\mathrm{diam\ } X + 1]$, there exists $\delta_1' > 0$ such that 
\begin{equation*}
0 \leq a ,b \leq \mathrm{diam\ }X+1,\ |a-b| < \delta_1' 
\quad \Rightarrow \quad
|\varphi^{-1}(a) - \varphi^{-1}(b) | < \frac{2\epsilon}{c_d (\mathrm{diam \ }Y + 1)}.
\end{equation*}
Let $X'$ be an FPDMS such that $d_{GH}(X,X') < \delta_2' :=  \min\{1, \frac{\delta_1'}{3}\}$. 
By \eqref{eq:GH-distort}, there exist two functions $f:X \rightarrow X'$ and $f':X'\rightarrow X$ such that
\begin{equation*}
\begin{aligned}\frac{1}{2} 
\max \Big\{ \sup_{x_1, x_2 \in X} 
& \big \{ \big|d_X(x_1, x_2) - d_{X'} \big( f(x_1 ), f(x_2) \big) \big| \big \},   \\
\sup _{x_1', x_2' \in X'}  
&\big\{ \big| d_X \big(f'(x_1' ), f'(x_2' ) \big) - d_{X'} (x_1' ,x_2' ) \big| \big\}, \\
 \sup_{x_1\in X, x_2' \in X'} 
&\big\{ \big| d_X \big(x_1, f'(x_2')\big) - d_{X'} \big(f(x_1), x_2'\big)\big| \big \}   \ \ 
\Big\} < \frac{\delta_1'}{3}.
\end{aligned}
\end{equation*}
This gives
\begin{equation*}
\begin{aligned} \frac{1}{2}
\max \Big\{ \sup_{y_1, y_2 \in Y} 
& \big \{ \big| d_Y(y_1, y_2) - d_{Y'} \big( (h'\circ f\circ h^{-1})(y_1 ), (h'\circ f\circ h^{-1})(y_2) \big) \big| \big \},   \\
\sup _{y_1', y_2' \in Y'}  
&\big\{ \big| d_Y \big((h\circ f'\circ h'^{-1})(y_1' ), (h\circ f'\circ h'^{-1})(y_2' ) \big) - d_{Y'} (y_1' ,y_2' ) \big| \big\}, \\
 \sup_{y_1\in Y, y_2' \in Y'} 
&\big\{ \big| d_Y \big(y_1, (h\circ f'\circ h'^{-1})(x_2')\big) - d_{Y'} \big((h' \circ f\circ h^{-1})(y_1), y_2'\big)\big| \big \} 
\Big\} \\
=\frac{1}{2}
\max \Big\{ \sup_{x_1, x_2 \in X} 
& \big \{ \big| \varphi^{-1}\big(d_X(x_1, x_2)\big) - \varphi^{-1}\Big( d_{X'} \big( f(x_1 ), f(x_2) \big)\Big) \big| \big \},   \\
\sup _{x_1', x_2' \in X'}  
&\big\{ \big| \varphi^{-1} \Big( d_X \big(f'(x_1' ), f'(x_2' ) \big) \Big) - \varphi^{-1} \big( d_{X'} (x_1' ,x_2' ) \big) \big| \big\}, \\
 \sup_{x_1\in X, x_2' \in X'} 
&\big\{ \big| \varphi^{-1} \Big( d_X \big(x_1, f'(x_2')\big) \Big) - \varphi^{-1} \Big( d_{X'} \big(f(x_1), x_2'\big)\big| \Big) \big \}   \ \ 
\Big\} \\
< \frac{\epsilon}{c_d (\mathrm{diam \ }Y + 1)},&
\end{aligned}
\end{equation*}
and therefore we have $d_{H,\mathrm{rigid}}(Y,Y') \leq c_d(\mathrm{diam\ } Y+1) \cdot d_{GH}(Y,Y') < \epsilon$.
\item This is proved in a similar way as above.
\end{enumerate}
\end{proof}

\begin{cor} \label{cor:X-vs-Y}
For a sequence $(X_n)$ of FPDMSs of cardinality $\leq d+1$ convergent in Gromov-Hausdorff distance, we can choose similarity embedding $Y_n$ of $X_n$ for each $n=1,2,\cdots$ such that $(Y_n)$ converges in Hausdorff distance. 
Conversely, given sequence $(Y_n)$ of similarity embeddings in $\bR^d$ convergent in Hausdorff distance, the sequence $(X_n)$ of FPDMSs converges in Gromov-Hausdorff distance.
\end{cor}


\section{Main Result} \label{sec:type-wise}

It was proved \cite{Katsumasa+2025} that magnitude is nowhere continuous in the set of (isometry classes of) finite metric spaces. 
In other words, given metric space $X$ we cannot arbitrarily lessen the difference $\big| |X|-|X'|\big|$ by controlling $d_{GH}(X,X')$ \textit{only}, for another metric space $X'$. 
As a next step, we re-examine the same problem with an additional restriction that $X'$ \textit{remains in a specific subset} of a neighbor of $X$. 
Indeed, there exists a canonical partition of any sufficiently ball in Gromov-Hausdorff distance, as explained below.

Let $X$ be a finite metric space and $\delta_0 := \min\{d(x_1,x_2): x_1,x_2 \in X,\ x_1 \neq x_2\}$. 
If $X'$ is another finite metric space such that $d_{GH}(X,X') < \varepsilon <  \frac{1}{4}\delta_0$, then there exist isometries $\phi:X \rightarrow Z$ and $\phi': X' \rightarrow Z$ into another metric space $Z$ such that $d_H(\phi(X), \phi(X')) < \frac{1}{4} \delta_0$. 
Then the projection $\pi:\phi(X') \rightarrow \phi(X)$ is well-defined by declaring $\pi(z')$ to be the closest element of $\phi(X)$ from $z'$, and we have
\begin{align*}
d_{X'} (x_1',x_2' )
&= d_Y (\phi(x_1'), \phi(x_2') ) \\
&\begin{cases}
< 2\varepsilon < \frac{1}{2} \delta_0, & \text{if } \pi (\phi(x_1') ) = \pi(\phi(x_2')) \\
> \frac{1}{2}\delta_0  , & \text{if } \pi (\phi(x_1') ) \neq \pi(\phi(x_2'))
\end{cases}
\end{align*}
Consequently, $X'$ can be partitioned into $k:=\#X$ subsets $X_1',\cdots,X_{k}'$ such that
$d_{X'} (X_i', X_j' ) > \frac{1}{2} \delta_0 $ for $i\neq j$ and $\mathrm{diam} X_i' < \frac{1}{2} \delta_0$ for each $i$. 
As a description of such ``clustering'' of $X'$, we will use the sequence obtained by rearranging the integers $\#X_1 -1, \#X_2 -1, \cdots, \#X_k -1$ to be monotonely decreasing. 
Hence the following definition:
\begin{defn}
Let $R$ be the set of all finite monotone decreasing sequence of nonnegative integers. 
A finite metric space $X$ is \textbf{clustered in type} $r = \langle r_1, r_2,\cdots, r_m\rangle \in R$ if there exists $\varepsilon > 0$ such that the relation $\{ (x_1, x_2) \in X \times X: d(x_1, x_2) < \epsilon\}$ is an equivalence relation on $X$, resulting in $m$ equivalent classes with respective cardinalities $r_1 +1 ,\cdots, r_m + 1$. 
\end{defn}
In terms of the above definition, we have just shown that every sufficiently small neighbor of $X$ can be partitioned into subsets of metric spaces clustered in different types in $R$. 
Therefore, it is natural to ask, as in Theorem \ref{thm:main}, whether $\lim_{n \rightarrow \infty} |X_n| = |X|$ if $\lim_{n \rightarrow \infty}X_n = X$ in Gromov-Hausdorff distance \textit{and} $X_n$ remains clustered in a certain type. 

Now we prepare for the proof of Theorem \ref{thm:main}. 
Because of the distance relation \eqref{eq:dist-x-y}, an FPDMS $X$ is clustered in type $r$ if and only if its similarity embedding $Y$ is clustered in type $r$. 
By applying Corollary \ref{cor:X-vs-Y}, it suffices to prove that the corresponding statement
\begin{quote}
If $(Y_n)$ is a sequence of similarity embeddings clustered in type $r$ converging in Hausdorff distance to another similarity embedding $Y$ of cardinality $k$, then $\lim_{n \rightarrow \infty} \rho_{Y_n} = \rho_Y$.
\end{quote}
is true if and only if $r$ and $k$ satisfy the same condition as in Theorem \ref{thm:main}.
We will present steps of proof for the ``if'' part and counterexamples for the ``only if'' part in the following subsections.

\subsection{Proof of the ``if'' part}
There exists a sequence $(Y_n)$ of Euclidean subsets in(not necessarily similarity embeddings) such that $\lim_{n \rightarrow \infty } Y_n = Y$ in $d_H$ but $\lim_{n\rightarrow \infty} \rho_{Y_n} = \rho_Y$. 
Indeed, if we let
\begin{equation} \label{eq:simple-counterexample}
Y_n = \left\{ (0,0),\ (1,0),\ \left( 1+\frac{1}{n} , \frac{1}{n} \right) \right\} ,\  Y = \big\{(0,0),\ (1,0)\big\} \subset \bR^2,
\end{equation}
the circumcenters are
\begin{equation*}
\mr K_{Y_n} = \left( \frac{1}{2}, \frac{1}{2}+\frac{1}{n} \right) , \quad \mr K_Y =  \left( \frac{1}{2} ,0 \right)
\end{equation*}
and hence $\lim_{n \rightarrow \infty} \rho_{Y_n} \neq \rho_Y$ (see Figure \ref{fig:triangle135}). 
Therefore, proof of the ``if'' part must depend on an additional constraint, the tri-similarity inequality \eqref{eq:strongly-acute}. 
The following lemma shows a consequence of the tri-similarity inequality.

\begin{figure}
\centering
\begin{tikzpicture}
\tikzmath{\scal = 2.5; \ncas = 8;
}
\draw (0,0) -- (\scal,0) -- ({\scal*(1+1/\ncas)},\scal*1/\ncas) -- (0,0);
\draw (\scal*0.5,{\scal*(0.5+1/\ncas)}) circle ({\scal*sqrt(0.25+0.25+1/\ncas+1/(\ncas*\ncas});

\draw[->,dotted] (\scal,0) -- ({\scal*(1+1/\ncas)},0);
\draw[->,dotted] ({\scal*(1+1/\ncas)},0) -- (\scal,0);
\draw ({\scal*(1+0.5/\ncas)},0) node[anchor=north]{$\frac{1}{n}$};

\draw[->,dotted] ({\scal*(1+1/\ncas)},0) -- ({\scal*(1+1/\ncas)},\scal*1/\ncas);
\draw[->,dotted] ({\scal*(1+1/\ncas)},\scal*1/\ncas) -- ({\scal*(1+1/\ncas)},0);
\draw  ({\scal*(1+1/\ncas)},\scal*0.5/\ncas) node[anchor=west]{$\frac{1}{n}$};

\draw[dashed] (\scal*0.5,0) circle(\scal*0.5);
\draw[fill] (\scal*0.5,{\scal*(0.5+1/\ncas)}) node[anchor=south]{$\mr K_{Y_n}$} circle (1pt);
\draw[fill] (\scal*0.5,0) node[anchor=north]{$\mr K_Y$} circle (1pt);
\end{tikzpicture}
\caption{Non-continuity of circumradii with respect to Hausdorff distance.}
\label{fig:triangle135}
\end{figure}
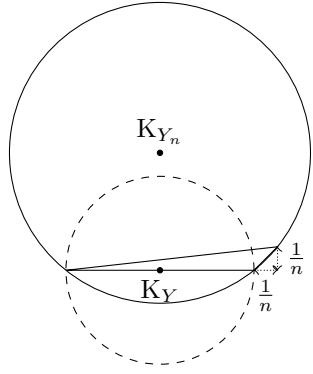

\begin{lem} \label{lem:non-flat}
If any triple of four points $y_1,y_1',y_2,y_2' \in \bR^3$ ($y_1 \neq y_1',\ y_2 \neq y_2'$) satisfies the tri-similarity inequality \eqref{eq:strongly-acute}, we have
\begin{equation*}
\left| \frac{y_1'-y_1}{\| y_1'-y_1\|} \cdot \frac{y_2'-y_2}{ \| y_2'-y_2 \| }  \right|
\\ \leq M \cdot  \min\left\{ \frac{\|y_1'-y_1\|}{\|y_2' -y_2\|}, \frac{\|y_2'-y_2\|}{\|y_1' -y_1\|} \right\} ,
\end{equation*}
where
\begin{equation*}
M = 1 - \frac{1}{2} \Big( \min \big\{ \| y_1 - y_2 \|,\ \| y_1' - y_2 \|,\ \| y_1 - y_2' \|,\ \| y_1' - y_2' \| \big\} \Big)^2
\end{equation*}
\end{lem}

\begin{proof}
We use the identity
\begin{equation} \label{eq:skew-inner-prod}
(y_1' -y_1) \cdot (y_2' -y_2) = \frac{1}{2} \big( \|y_2' -y_1 \| ^2  + \| y_2-y_1' \| ^2  - \| y_2 - y_1 \| ^2 - \| y_2' - y_1' \| ^2  \big).
\end{equation}
Applying tri-similarity inequalities
\begin{align*}
\|y_2' -y_1 \| ^2 
& \leq \| y_2' - y_1'\| ^2 + \| y_1' - y_1 \| ^2 - \frac{1}{2} \| y_2' -y_1' \| ^2 \| y_1' -y_1 \|^2  \\
\|y_2 - y_1' \|^2 
& \leq \| y_2 - y_1 \| ^2 + \| y_1 - y_1' \|^2 - \frac{1}{2} \| y_2 -y_1 \| ^2 \| y_1 - y_1' \| ^2 ,
\end{align*}
to equation \eqref{eq:skew-inner-prod}, we obtain
\begin{align*}
(y_1' - y_1) \cdot (y_2' - y_2) 
& \leq \| y_1' - y_1 \| ^2 \left( 1- \frac{1}{4}\|y_2'-y_1'\|^2 - \frac{1}{4} \| y_2 -y_1\|^2 \right) \\
& \leq M \cdot \| y_1' - y_1 \| ^2 .
\end{align*}
Likewise, applying tri-similarity inequalities
\begin{align*}
\|y_2' - y_1 \| ^2 
& \geq \| y_2' - y_1' \| ^2 - \| y_1' -y_1 \|^2 + \frac{1}{2} \| y_2' -y_1 \| ^2 \| y_1' -y_1 \|^2  \\
\| y_2 -y_1' \| ^2 
& \geq \| y_2 -y_1 \| ^2 - \| y_1 -y_1' \| ^2 + \frac{1}{2} \| y_2 -y_1' \| ^2 \| y_1 -y_1' \| ^2
\end{align*}
to equation \eqref{eq:skew-inner-prod}, we obtain
\begin{align*}
(y_1' - y_1) \cdot (y_2' - y_2) 
& \geq - \| y_1' -y_1\| ^2 \left( 1 - \frac{1}{4} \| y_1' -y_1 \|^2 - \frac{1}{4} \| y_2 -y_1' \|^2  \right) \\
& \geq - M \cdot \| y_1' - y_1 \| ^2
\end{align*}
Thus, we conclude 
\begin{equation*}
\big| (y_1' - y_1) \cdot (y_2' - y_2) \big| \leq M \cdot \| y_1' -y_1 \| ^2.
\end{equation*}
From a similar computation we obtain
\begin{equation*}
\big| (y_1' - y_1) \cdot (y_2' - y_2) \big| \leq M \cdot \| y_2' -y_2 \| ^2,
\end{equation*}
and the proof is complete by combining the above two results.
\end{proof}

In particular, the above lemma implies that as $\| y_1 -y_1' \|$ approaches zero, the angle between the vectors $y_1'-y_1$ and $y_2'-y_2$ approaches the right angle. 
This idea lies behind the proof of the following proposition.

\begin{prop} \label{prop:R-conv+2}
Let $Y' = \{ y_1' ,\cdots, y_{k+2}'\},\ Y = \{y_1,\cdots, y_k\} \subset \bR^{k+1}$ be similarity embeddings with circumradii $\rho_{Y'}$ and $\rho_{Y}$, respectively. 
For each $\alpha,\beta \in \{1,\cdots,k\}$, let
\begin{equation*}
h_{\alpha\beta} := \sqrt{ \|y_1' -y_1\|^2 + \cdots + \|y_k' -y_k\|^2 + \|y_{k+1}'-y_{\alpha}\|^2 + \|y_{k+2}'-y_{\beta}\|^2 }.
\end{equation*}
Then we have
\begin{equation*}
| {\rho_{Y'}} - \rho_Y | \leq C_Y h_{\alpha \beta}
\end{equation*}
for all sufficiently small $h_{\alpha\beta}$, where $C_Y$ is a constant depending on $Y$.
\end{prop}

\begin{proof}
Let $\mr K_{Y'}$ and $\mr K_Y$ be the circumcenters of $Y'$ and $Y$, respectively. 
For each fixed $\mr K_Y$, the vector $\mr K_{Y'}-\mr K_Y$ can be determined by its perpendicularity to the affine subspace $\mathrm{aff\ }Y_k'$ spanned by $Y_k' := \{y_1',\cdots,y_k'\}$ and equidistance conditions
\begin{align*}
\|\mr K_{Y'} - y_{k+1}' \| ^2 & =\|\mr K_{Y'} - y_{\alpha}'\|^2 \\
\|\mr K_{Y'} - y_{k+2}' \| ^2 &= \|\mr K_{Y'} - y_{\beta}' \|^2.
\end{align*}
Let $\mathbf{u}'$ and $\mathbf{v}'$ be the orthogonal projection of $\mathbf{u} := \frac{y_{k+1}'-y_\alpha'}{\|y_{k+1}'-y_\alpha'\|}$ and $\mb{v}:=\frac{y_{k+2}'-y_\beta'}{\|y_{k+2}'-y_\beta'\|}$, respectively, onto the orthogonal complement of $\mathrm{aff\ }Y_k'$ in $\bR^{k+1}$. 
Then we can paraphrase the above equations as
\begin{align} \label{eq:circumcent-1}
\mb{u}' \cdot (\mr K_{Y'}-\mr K_Y) 
& = \mb{u} \cdot \left( \frac{y_{k+1}'+y_\alpha'}{2} -\mr K_Y \right) \\ \label{eq:circumcent-2}
\mb{v}' \cdot (\mr K_{Y'}-\mr K_Y) 
& = \mb{v} \cdot \left( \frac{y_{k+2}'+y_\beta'}{2} - \mr K_Y \right).
\end{align} 
Our strategy is to estimate the norm $\|{\mr K}_{Y'}-{\mr K}_Y\|$ of the solution of equations \eqref{eq:circumcent-1}-\eqref{eq:circumcent-2} in terms of coefficient vectors $\mb{u}'$ and $\mb{v}'$ and right-hand sides.

We deal with the coefficient vectors first. 
If each $y_i'$ is close enough to $y_i'$ for each $i\in\{1,\cdots,k\}$, by Lemma \ref{lem:non-flat} we have
\begin{align*}
\left| \mb{u}  \cdot \frac{y_j'-y_i'}{\|y_j'-y_i'\|} \right|
& \leq  \frac{\| y_{k+1}'-y_\alpha'\|}{\|y_j' - y_i'\| } \\
& \leq \frac{\| y_{k+1}'-y_\alpha'\|}{\|y_j - y_i\| - \|y_j'- y_j\| - \| y_i' - y_i \| }  \\
& \leq C_Y^{(0)} h_{\alpha \beta}
\end{align*}
and similar estimate for $\mb{v}$ 
(Hereafter, each symbol $C_Y^{(j)}$ will refer to a positive constant depending on $Y$). 
Since vectors $\{y_i'-y_k':i=1,\cdots,k-1 \}$ span the tangent space of $\mathrm{aff\ }Y_k'$, each $y_i' - y_k'$ is close enough to $y_i-y_k$ when $h_{\alpha\beta}$ is small enough and vectors $\{y_i-y_k:i=1,\cdots,k-1 \}$ are linearly independent,
we have
\begin{equation*}
\left\| \mb{u} - \mb{u}' \right \| , \| \mb{v} -\mb{v}'\| \leq C_Y^{(1)} h_{\alpha \beta}.
\end{equation*}
Then,
\begin{equation}\label{eq:circumcent-coef-estimate}
\begin{aligned}
\| \mb{u}' \wedge \mb{v}' \| 
& \geq \| \mb{u} \wedge \mb{v} \| - C_Y^{(2)} h_{\alpha\beta} \\
& = \big[1 - (\mb{u} \cdot \mb{v}) ^2  \big]^{\frac{1}{2}} - C_Y^{(2)} h_{\alpha\beta} \\
& \geq C_Y^{(3)} - C_Y^{(2)} h_{\alpha \beta},
\end{aligned}
\end{equation}
where the inequality on the last line is obtained by Lemma \ref{lem:non-flat}.

Next, we estimate right-hand sides of equations \eqref{eq:circumcent-1} and \eqref{eq:circumcent-2}. 
We can write $\mr K_Y  = c_1 y_1 + \cdots + c_k y_k$ for some real numbers $c_1,\cdots,c_k$ such that $c_1 + \cdots + c_k = 1$ and 
\begin{align*}
\frac{y_{k+1}'+y_{\alpha}'}{2} - \mr K_Y 
& = \frac{y_{k+1}'+y_{\alpha}'}{2} - (c_1 y_1' + \cdots + c_k y_k')  \\
&\phantom{{}={}\qquad\qquad} + \big[ c_1 (y_1' -y_1) + \cdots + c_k (y_k' - y_k) \big] \\
& = c_1 ' (y_{k+1}' - y_1') + \cdots + c_k' (y_{k+1}' - y_k') \\
& \phantom{ {}={}\qquad\qquad} + \big [ c_1(y_1' - y_1 ) + \cdots + c_k (y_k' - y_k)  \big]
\end{align*}
Then, by Lemma \ref{lem:non-flat} we have
\begin{equation} \label{eq:alphabet-intercept-estimate}
\begin{aligned}
\left| \mb{u} \cdot \left( \frac{y_{k+1}'+y_\alpha'}{2} -\mr K_Y \right)\right|
& \leq |c_1'|\ | \mb{u} \cdot (y_{k+1}' -y_1') | + \cdots + |c_k'|\ |\mb{u}\cdot (y_{k+1}' -y_k')| \\
& \phantom{\qquad\qquad} + \max_{1 \leq i \leq k} |c_i| \cdot  h_{\alpha\beta} \\
& \leq (|c_1'| + \cdots + |c_k'| ) \|y_{k+1}'-y_\alpha' \|  + \max_{1 \leq i \leq k} |c_i| \cdot  h_{\alpha\beta} \\
& \leq C_Y^{(4)} h_{\alpha \beta}
\end{aligned}
\end{equation}
and a similar estimate for $\mb{v}$. 

By estimates \eqref{eq:circumcent-coef-estimate} and \eqref{eq:alphabet-intercept-estimate}, the solution $\mr K_{Y'}-\mr K_Y$ of \eqref{eq:circumcent-1}-\eqref{eq:circumcent-2} satisfy $\|\mr K_{Y'}-\mr K_Y\| \leq C_Y^{(5)} h_{\alpha \beta}$. 
Finally, we conclude that
\begin{align*}
| {\rho_{Y'}} - \rho_Y |
= \big| {\|\mr K_{Y'}-y_1'\|} - {\|\mr K_Y-y_1\|}  \big|
\leq  {\|\mr K_{Y'}-\mr K_Y\| + \|y_1'-y_1\|}
\leq  C_Y h_{\alpha\beta}.
\end{align*}
\end{proof}

The above proposition corresponds to $\|r\|_1=2$ case of Theorem \ref{thm:main}. 
The ``exceptional'' case $\|r\|_1 =3$ and $\#X=1$ is explained through the following geometric observation.

\begin{prop} \label{prop:4pts-radius}
If any angle determined by three points in a $4$-point set $Y =\{\mr A,\mr B,\mr C,\mr D\} \subset \bR^3$ is acute, then the circumradius $\rho_Y$ satisfies
\begin{equation*}
\rho_Y \leq 2\, \mathrm{diam}\, Y.
\end{equation*}
In particular, the above inequality holds for any $4$-point similarity embedding.
\end{prop}

\begin{proof}
Let us assume on the contrary that $\rho := \rho_Y > 2\,\mathrm{diam}\, Y$. 
By applying a suitable rigid transformation if necessary, we may assume that $Y$ is a subset of the sphere $x^2 +y^2 +z^2 = \rho^2$ and the coordinates are
\begin{equation*}
\mr A = (0,a,b),
\ \mr B = (0,-a,b), 
\ \mr C = (x_0, y_0, z_0) ,
\ \mr D = (x_1, y_1, z_1),
\end{equation*}
where
\begin{equation*}
\begin{gathered}
0<a< \frac{\rho}{4},\quad b=\sqrt{\rho^2 - a^2 }  > \frac{\rho}{2},\\
x_0 = \sqrt{\rho^2 - y_0^2 - z_0^2}, \quad x_1 = \pm \sqrt{ \rho^2 - y_1^2 - z_1^2}
.
\end{gathered}
\end{equation*}
Since $\mr A\mr C,\mr B\mr D \leq \mathrm{diam\ } Y< \frac{\rho}{2}$ and the angles $\mr A\mr C\mr B,\mr A\mr B\mr C,\mr A\mr D\mr B$ and $\mr A\mr B\mr D$ are acute, we have additional restrictions
\begin{equation*}
|y_0|,|y_1| < a,\quad 0< z_0,z_1 < b,
\end{equation*}
which implies $x_1 \neq 0$ in particular.
The proof proceeds by casework on the sign of $x_1$.

\textit{Case 1}. $x_1 > 0$. 
Let us consider the planar regions 
\begin{align*}
S_{+} &= \Big\{(y,z)\in \bR^2 : \mr E = (x,y,z),\ x  = \sqrt{\rho^2 - y^2 - z^2}, \ \overrightarrow{\mr E\mr A} \cdot \overrightarrow{\mr E\mr C} \leq 0 \Big\} \\
& = \{(y,z) \in \bR^2 : f_{a,\mr C} (y,z) \geq 0  \} \\
S_{-} &= \Big\{(y,z)\in \bR^2 : \mr E = (x,y,z),\ x  = \sqrt{\rho^2 - y^2 - z^2}, \ \overrightarrow{\mr E\mr B} \cdot \overrightarrow{\mr E\mr C} \leq 0 \Big\} \\
& = \{(y,z) \in \bR^2 : f_{-a,\mr C} (y,z) \geq 0  \},
\end{align*}
where
\begin{equation} \label{eq:aux-f-case1}
f_{\pm a,\mr C}(y,z) =  x_0^2 (\rho^2 -y^2  -z^2 ) - \big[(\pm a+y_0)(\pm a-y) + (b+z_0) (b-z) \big]^2  
\end{equation}
Thus, each of $S_{\pm}$ is equal to the union of an ellipse and its interior. 
Since we can put $\mr E=\mr A$ and $\mr E=\mr C$ in the definition of $S_{+}$ we have $(a,b),(y_0,z_0) \in S_{+}$. 
We also have $(a,z_0) \in S_+$ because
\begin{align*}
f_{a,\mr C}(a,z_0) 
& = x_0^2 (b^2 - z_0^2 )- \big[(b+z_0)(b-z_0)\big]^2 \\
&= (x_0^2 -b^2 +z_0^2 )(b^2 -z_0^2 ) \\
& = (a^2 - y_0^2 )(b^2 -z_0^2 ) \geq 0
\end{align*}
and $(u_+,b)\in S_+$ where $u_+:= \frac{ (y_0+a)^2 -x_0^2 }{(y_0+a)^2 +x_0^2}a$ because $f_{a,\mr C}\left(u_{+}, b \right)=0$. 
By convexity, $S_+$ contains the trapezoid with vertices $(a,b),\,(a,z_0),\,(y_0,z_0),\, (u_+,b)$. 
Likewise, $S_-$ contains the trapezoid with vertices $(-a,b),\, (-a,z_0),\, (y_0,z_0),\, (u_-,b)$, where $u_-:= \frac{ (y_0-a)^2 -x_0^2 }{(y_0-a)^2 +x_0^2}(-a)$. 
Because
\begin{align*}
u_{+} - u_{-} 
&= \frac{\big[(a+y_0 )^2 -x_0^2 \big]\big[(a-y_0)^2 +x_0^2 \big] + \big[(a-y_0)^2 - x_0^2\big]\big[(a+y_0)^2 +x_0^2\big]}{\big[(a+y_0)^2 +x_0^2 \big] \big[(a-y_0)^2 +x_0^2\big]}  \cdot a\\
& = \frac{  2(a+y_0 )^2 (a-y_0 )^2 - 2 x_0^4    }{\big[(a+y_0)^2 +x_0^2 \big] \big[(a-y_0)^2 +x_0^2\big]} \cdot a\\
& = \frac{  2(a^2 -y_0 ^2 + x_0^2 )(a^2 - y_0 ^2 - x_0 ^2 ) }{\big[(a+y_0)^2 +x_0^2 \big] \big[(a-y_0)^2 +x_0^2\big]} \cdot a \\
& = \frac{  2(a^2 -y_0 ^2 + x_0^2 )(z_0^2 - b^2  ) }{\big[(a+y_0)^2 +x_0^2 \big] \big[(a-y_0)^2 +x_0^2\big]} \cdot a  < 0,
\end{align*}
the union $S_{+} \cup S_{-}$ contains entire rectangle $|y| \leq a,\ z_0 \leq z \le b$. 
Since both angles $\rm ADC$ and $\rm BDC$ are acute, $(y_1,z_1) \not\in S_{+} \cup S_{-}$, which implies $z_1 <z_0$. 
The same argument with switching the roles of $\mr C$ and $\mr D$ gives the opposite result $z_0 < z_1$, which gives a contradiction. 

\textit{Case 2.} $x_1 < 0$. 
Let us consider planar regions
\begin{align*}
T_{+} 
& = \Big\{(y,z) \in \bR^2 : \mr E = (x,y,z),\ x=-\sqrt{\rho^2 - y^2 -z^2 },\ \overrightarrow{\mr A\mr C}\cdot \overrightarrow{\mr A\mr E} \leq 0  \Big\}  \\
& = \{(y,z) \in \bR^2 : f_{a,-\mr C}(y,z) \geq 0  \} \\
T_{-} 
& = \Big\{(y,z) \in \bR^2 : \mr E = (x,y,z),\ x=-\sqrt{\rho^2 - y^2 -z^2 },\ \overrightarrow{\mr B\mr C}\cdot \overrightarrow{\mr B\mr E} \leq 0  \Big\}  \\
& = \{(y,z) \in \bR^2 : f_{-a,-\mr C}(y,z) \geq 0  \} 
\end{align*}
where
\begin{align*}
f_{\pm a,-\mr C}(y,z) = x_0^2 (\rho^2 - y^2 -z^2 ) - \big[ (y_0 \mp a)(y \mp a) + (z_0 -b)(z-b)\big]^2.
\end{align*}
Note that the similarity with \eqref{eq:aux-f-case1} follows from $\overrightarrow{\mr A\mr C}\cdot \overrightarrow{\mr A\mr E} = \overrightarrow{\mr A\mr C}\cdot \overrightarrow{(-\mr C)\mr E} = \overrightarrow{\mr A\mr E}\cdot \overrightarrow{(-\mr C)\mr E}$. 
By a similar argument as in the previous step, the union $T_+ \cup T_-$ contains the entire rectangle $|y| \leq a,\ -z_0 \leq z \leq b$. 
Since both angles $\rm CAD$ and $\rm CBD$ must be acute, $(y_1,z_1) \not\in T_+ \cup T_- $, which implies $z_1 < -z_0 < 0$. 
This contradicts $z_1 > 0$.
\end{proof}

\begin{proof}[Proof of Theorem \ref{thm:main}, ``if'' part]
This follows from Proposition \ref{prop:R-conv+2} (if $\|r\|_1 \leq 2$) and Proposition \ref{prop:4pts-radius} (if $\|r\|_1=3$ and $\#X =1$), by Corollary \ref{cor:X-vs-Y}.
\end{proof}

\subsection{Proof of the ``only if'' part} \label{subsec:counterexamples}

We first examine a sufficient condition for $\lim_{n\rightarrow \infty} \rho_{Y_n} = \rho_Y$ to hold. 
Suppose that similarity embeddings $Y_n$ $(n=1,2,\cdots)$ and $Y$ are of the form
\begin{align*}
Y_n &= \{ y_n ^{(i,j)} : 1 \leq i \leq k ,\ 0 \leq j \leq r_i \} \quad (n=1,2,\cdots)\\
Y &=  \{ y^{(1)} ,  y^{(2)} , \cdots, y^{(k)} \},
\end{align*}
where $\lim_{n \rightarrow \infty} y_n^{(i,j)} = y^{(i)}$ for each $i=1,2,\cdots,k$ and $j=0,1,\cdots,r_i$. 
Then the circumcenter $\mr K_n := \mr K_{Y_n}$ of $Y_n$ is determined by equations
\begin{align*}
\|\mr K_n - y_n ^{(1,0)} \| 
&= \| \mr K_n - y_n^{(i,0)} \| \quad (2 \leq i \leq k) \\
\| \mr K_n - y_n^{(i,0)}\| 
& = \| \mr K_n - y_n^{(i,j)} \| \quad (1\leq i \leq k,\ 1 \leq j \leq r_i),
\end{align*}
or equivalently by
\begin{align*}
(y^{(1,0)}_n - y_n^{(i,0)}) \cdot (\mr K_n -\mr K)
& = \frac{\| y_n ^{(1,0)} -\mr K\|^2 - \|y_n^{(i,0)} -\mr K \|^2 }{2}  \quad (2 \leq i \leq k) \\
\frac{y_n^{(i,0)} - y_n^{(i,j)}}{\|y_n^{(i,0)} - y_n^{(i,j)} \|} \cdot (\mr K_n -\mr K)
& = \frac{\| y_n^{(i,0)} -\mr K \|^2 - \| y_n^{(i,j)}-\mr K \|^2}{2\|y_n^{(i,0)} - y_n^{(i,j)} \|} \quad (1\leq i \leq k,\ 1 \leq j \leq r_i)
\end{align*}
where $\mr K = \mr K_Y$ is the circumcenter of $Y$. 
Because the terms on the right-hand side converge to zero as $n \rightarrow \infty$, the condition
\begin{multline} \label{eq:vol-non-flat}
\liminf_{n \rightarrow \infty} \left| \mathrm{Vol} \left( 
y_n^{(1,0)} - y_n^{(2,0)},\cdots,y_n^{(1,0)} - y_n^{(k,0)}, 
\frac{y_n^{(1,0)} - y_n^{(1,1)}}{\|y_n^{(1,0)} - y_n^{(1,1)} \|} ,\cdots,
\frac{y_n^{(i,0)} - y_n^{(i,r_i)}}{\|y_n^{(i,0)} - y_n^{(i,r_i)} \|}
\right) \right| \\
= \liminf_{n \rightarrow \infty}  \left| \mathrm{Vol} \left(
y^{(1)} - y^{(2)},\cdots,y^{(1)} - y^{(k)}, 
\frac{y_n^{(1,0)} - y_n^{(1,1)}}{\|y_n^{(1,0)} - y_n^{(1,1)} \|} ,\cdots,
\frac{y_n^{(i,0)} - y_n^{(i,r_i)}}{\|y_n^{(i,0)} - y_n^{(i,r_i)} \|}
\right) \right| > 0.
\end{multline} 
implies $\lim_{n \rightarrow \infty} \|\mr  K_n -\mr  K \| = 0$ and equivalently $\lim_{n \rightarrow \infty} \rho_{Y_n} = \rho_Y$.
The condition \ref{eq:vol-non-flat} holds, for example, when vectors $y_n^{(i,0)}-y_n^{(i,j)}$ are pairwise orthogonal and orthogonal to $\mathrm{aff\ }Y$. 
The condition \ref{eq:vol-non-flat} means in general that the parallelopiped formed by vectors $y_n^{(i,0)}-y_n^{(i,j)}$ is kept away from two conditions: being parallel to $\mathrm{aff\ } Y$ and being degenerate. 
The former is always avoided because the angle between each $y_n^{(i,0)}-y_n^{(i,j)}$ and $\mathrm{aff\ }Y$ approaches the right angle, as mentioned immediately after Lemma \ref{lem:non-flat}. 
Then it must be the second condition that fails in the examples such that $\lim_{n \rightarrow \infty} \rho_{Y_n} \neq \rho_Y$. 
Thus, we attempt to search for counterexamples by taking a sequence $(Y_n)$ converging to $Y$ in a way such that the parallelopiped generated by vectors $y_n^{(i,0)}-y_n^{(i,j)}$ become ``gradually flat.'' 

In the following examples, we introduce a family $(Y_t)_{t>0}$ parametrized by positive real numbers. 
Each $Y_t$ will be a similarity embedding for sufficiently small $t>0$, but we defer the verification of the tri-similarity inequality to Appendix \ref{sec:computation}. 
We will denote by $\mr K_t  = \mr K_{Y_t}$ and $\mr K = \mr K_Y$ the circumcenters of $Y_t$ and $Y$, respectively, and by $\rho_t := \rho_{Y_t}$ and $\rho = \rho_Y$ the cirumradii of $Y_t$ and $Y$, respectively. 

\begin{ex} \label{ex:counter-(4)}
Here is a family $(Y_t)_{0<t < \delta}$ of $5-$point similarity embeddings clustered in type $\langle 4\rangle$ and a singleton similarity embedding $Y$ such that $\lim_{t\searrow 0}{Y_t = Y}$ in Hausdorff distance but $\lim_{t \searrow 0} \rho_{t} > \rho$.

Let $0 < s < \frac{6}{7}$  and $Y_t = \{\mr A,\mr B,\mr C,\mr D,\mr E\}$, where
\begin{align*}
&\mr A  = (-4t,0,0,0), 
\qquad \mr B=(2t,2{\sqrt{3}} t, 0,0),
\qquad \mr C = (2t,-2{\sqrt{3}} t, 0,0), \\
&\mr D  = (0,0,{3} t \sqrt{1- (st)^2 } , 3 s t^2 ) ,
\qquad\quad \mr E = (0,0,-{3} t \sqrt{1- (st)^2 }, 3 s t^2 ).
\end{align*}
Note that as $t \searrow 0$ all the points of $Y_t$ converge to the origin. 
On the other hand, the circumcenter and the circumradius of $Y_t$ are
\begin{align*}
\mr K_t = \left( 0,0,0, -\frac{7}{6}s \right) ,\qquad 
\rho_t = \sqrt{ 16t^2  + \frac{49}{36} s^2 },
\end{align*}
so we have $\lim_{t \searrow 0} \rho_t = \frac{7}{6} s > 0 = \rho$.
\end{ex}

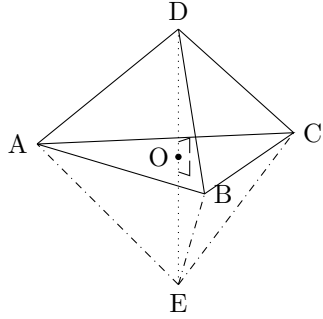
\begin{figure}
\centering
\begin{tikzpicture}
\tikzmath{\radiusa = 2.; \radiusb=0.5; \thetaA = 160;
\hD = 1.7;
}
\draw ({\radiusa*cos(\thetaA)},{\radiusb*sin(\thetaA)}) node[anchor=east]{$\mr A$}
-- ({\radiusa*cos(\thetaA+120)},{\radiusb*sin(\thetaA+120)}) node[anchor=west]{$\mr B$}
-- ({\radiusa*cos(\thetaA+240)},{\radiusb*sin(\thetaA+240)}) node[anchor=west]{$\mr C$}
-- ({\radiusa*cos(\thetaA)},{\radiusb*sin(\thetaA)});
\draw (0,\hD) node[anchor=south]{$\mr D$} -- ({\radiusa*cos(\thetaA)},{\radiusb*sin(\thetaA)});
\draw (0,\hD) -- ({\radiusa*cos(\thetaA+120)},{\radiusb*sin(\thetaA+120)});
\draw (0,\hD) -- ({\radiusa*cos(\thetaA+240)},{\radiusb*sin(\thetaA+240)});
\draw[dashdotted] (0,-\hD) node[anchor=north]{$\mr E$} -- ({\radiusa*cos(\thetaA)},{\radiusb*sin(\thetaA)});
\draw[dashdotted] (0,-\hD) -- ({\radiusa*cos(\thetaA+120)},{\radiusb*sin(\thetaA+120)});
\draw[dashdotted] (0,-\hD) -- ({\radiusa*cos(\thetaA+240)},{\radiusb*sin(\thetaA+240)});

\draw[dotted] (0,\hD) -- (0,-\hD);
\draw[fill] (0,0) node[anchor=east]{$\mr O$} circle (1pt);
\draw (0,0.2) -- (0.15,0.25) -- (0.15,0.05);
\draw (0,-0.2) -- (0.15,-0.25) -- (0.15,-0.05);
\end{tikzpicture}
\caption{The elements of $Y_t$ in Example \ref{ex:counter-(4)}. 
Three points $\mr D, \mr E$ and the origin $\mr O$ are not colinear.}
\end{figure}

\begin{ex} \label{ex:counter-(3)}
Here is a family $(Y_t)_{0<t < \delta}$ of $5-$point similarity embeddings clustered in type $\langle 3,0\rangle$ and a $2-$point similarity embedding $Y$ such that $\lim_{t\searrow 0}{Y_t = Y}$ in Hausdorff distance but $\lim_{t \searrow 0} \rho_{t} > \rho$. 

Let us put $\rho<\frac{1}{\sqrt{2}}$, and let $s$ be a positive real number such that $0<\cos s < \frac{2-4\rho^2}{8\rho}$ and $Y_t=\{\mr A,\mr B,\mr C,\mr D,\mr E\} \subset \bR^4$, where
\begin{align*}
\mr A &= (\rho,0,0,0) + ( -t^2 \cos s, -t^2 \sin s, \phantom{-}t\sqrt{1-t^2},\phantom{-\sqrt{1-t^2}}0), \\
\mr B &= (\rho,0,0,0) + ( + t^2 \cos s , \phantom{-} t^2 \sin s , \phantom{-\sqrt{1-t^2}}0,\phantom{-}t\sqrt{1-t^2 }), \\
\mr C &= (\rho,0,0,0) + (-t^2 \cos s,-t^2 \sin s,-t\sqrt{1-t^2},\phantom{-\sqrt{1-r^2}}0),\\
\mr D &= (\rho,0,0,0) + ( + t^2 \cos s , \phantom{-} t^2 \sin s ,\phantom{-\sqrt{1-r^2}}0,-t\sqrt{1-t^2 }), \\
\mr E &= (-\rho,0,0,0).
\end{align*}
Note that as $t\searrow 0$ four points $\rm A,B,C$ and $\mr D$ converges to the origin $\mr O=(0,0,0,0)$ so that $\lim_{t \searrow 0} Y_t = Y$ in Hausdorff distance to $Y=\{\mr O,\mr E\}$ and $\rho$ is indeed equal to the circumradius of $Y$.
On the other hand, the circumcenter and the circumradius of $Y_t$ are
\begin{align*}
\mr K_t &= \left(\frac{t^2}{4\rho},\frac{4\rho^2-t^2}{4\rho\tan s},0,0 \right), \\
\rho_t &=  \sqrt{\frac{\rho^2}{\sin^2 s}  - \frac{1}{2 \sin^2 s } t^2 + \frac{1}{4\rho^2 \sin ^2 s } t^4 },
\end{align*}
so we have $\lim_{t \searrow 0} \rho_t = \frac{\rho}{\sin s} > \rho$.
\end{ex}

\begin{figure}
\centering
\begin{tikzpicture}
\tikzmath{\radiusa = 2.; \radiusb=0.5; 
\radiusc = -1.99; \tiltc=0.5; \radiusf=sqrt(\radiusa*\radiusa-\radiusc*\radiusc);
\radiuse=0.5; \tilte=0.8; \radiusd=sqrt(\radiusa*\radiusa-\radiuse*\radiuse);
\arm = 6;
\thetaA=-50; \thetaC=110;
}
\draw (0,0) circle (\radiusa);
\draw[dotted,domain=0:180,samples=20] plot ({\radiusa*cos(\x)},{\radiusb*sin(\x)});
\draw[dashed,domain=180:360,samples=20] plot ({\radiusa*cos(\x)},{\radiusb*sin(\x)});
\draw[dashdotted] (0,0) -- (\arm,\arm*0.1) node[anchor=west]{$\mr E$};
\draw[fill] (\arm,\arm*0.1) circle(1pt);

\draw[dashed,domain=0:180,samples=20] plot ({\radiusf*cos(\x)-\tiltc*\radiusc*sin(\x)},{\tiltc*\radiusf*sin(\x)+\radiusc*cos(\x)});
\draw[dotted,domain=180:360,samples=20] plot ({\radiusf*cos(\x)-\tiltc*\radiusc*sin(\x)},{\tiltc*\radiusf*sin(\x)+\radiusc*cos(\x)});

\draw[dashed,domain=0:180,samples=20] plot ({\radiuse*cos(\x)-\tilte*\radiusd*sin(\x)},{\tilte*\radiuse*sin(\x)+\radiusd*cos(\x)});
\draw[dotted,domain=180:360,samples=20] plot ({\radiuse*cos(\x)-\tilte*\radiusd*sin(\x)},{\tilte*\radiuse*sin(\x)+\radiusd*cos(\x)});

\draw[fill] ({\radiusa*cos(\thetaA)},{\radiusb*sin(\thetaA)}) node[anchor=west]{$\mr A$} circle (1pt);
\draw[fill] ({\radiusa*cos(\thetaC)},{\radiusb*sin(\thetaC)}) node[anchor=west]{$\mr C$} circle (1pt);
\draw[fill] ({\radiusa*(-0.17)},{\radiusa*0.93}) node[anchor=east]{$\mr B$} circle(1pt);
\draw[fill] ({\radiusa*(-0.21)},{-\radiusa*0.96}) node[anchor=east]{$\mr D$} circle(1pt);
\end{tikzpicture}
\caption{The elements of $Y_t$ in Example \ref{ex:counter-(3)}. 
Four points $\mr A, \mr B, \mr C$ and $\mr D$ are on a common $2-$sphere centered at the origin and on a common $3-$dimensional hyperplane on which $\mr E$ is not.}
\end{figure}
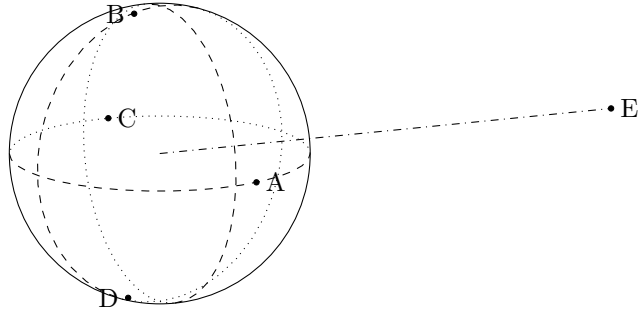

\begin{ex} \label{ex:counter-(2,1)}
Here is a family $(Y_t)_{0<t < \delta}$ of $5-$point similarity embeddings clustered in type $\langle2,1 \rangle$ and a $2-$point similarity embedding $Y$ such that $\lim_{t\searrow 0}{Y_t = Y}$ in Hausdorff distance but $\lim_{t \searrow 0} \rho_{t} > \rho$.

Let us put $\rho < \frac{1}{\sqrt{2}}$, and let $s>0$ and $Y_t = \rm \{A,B,C,D,E\} \subset \bR^4$, where
\begin{align*}
\mr A &= (-\rho,0,0,0),& \mr B&=\mr A+(t^2, \sqrt{3}t, t, st^2 ), &\mr C= \mr A+(t^2,\sqrt{3}t ,-t,st^2 )  \\
\mr D &= (\rho,0,0,0),& \mr E & = \mr D +(-t^2 , 2t,0,st^2 ).
\end{align*}
Note that as $t \searrow 0$ two points $\mr B,\mr C$ converge to $\mr A$ and point $\mr E$ converges to $\mr D$, so that $\lim_{t \searrow 0} Y_t = Y = \rm \{A,D\}$ in Hausdorff distance and $\rho$ is indeed equal to the circumradius of $Y$.

On the other hand, the circumcenter and the circumradius of $Y_t$ are
\begin{align*}
\mr K_t &= \left(0,0,0,\frac{2-\rho}{s} + \frac{1+s^2}{2s}t^2\right) ,
\\ \rho_t &= \sqrt{\rho^2 + \left( \frac{2-\rho}{s} + \frac{1+s^2}{2s}t^2 \right)^2 },
\end{align*}
so we have $\lim_{t \searrow 0} \rho_t = \sqrt{ \rho^2 + \left(\frac{2-\rho}{s}\right)^2 } > \rho$. 
\end{ex}

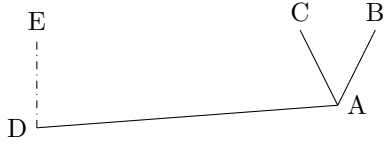
\begin{figure}
\centering
\begin{tikzpicture}
\tikzmath{\dx = 4; \dy=0.3; \bx=0.5; \by=1.0;
}
\draw (0,0) node[anchor=west]{$\mr A$} -- (-\dx,-\dy) node[anchor=east]{$\mr D$} ;
\draw (0,0) -- (\bx,\by) node[anchor=south]{$\mr B$};
\draw (0,0) -- (-\bx,\by) node[anchor=south]{$\mr C$};
\draw[dashdotted] (-\dx,-\dy) -- (-\dx,-\dy+1.18*\by) node[anchor=south]{$\mr E$};
\end{tikzpicture}
\caption{The elements of $Y_t$ in Example \ref{ex:counter-(2,1)}. 
Four points $\mr A, \mr B, \mr C$ and $\mr D$ are on a common $3-$dimensional hyperplane on which $\mr E$ is not.}
\end{figure}

\begin{ex} \label{ex:counter-(1,1,1)}
Here is a family $(Y_t)_{0<t <\delta}$ of $6-$point similarity embeddings clustered in type $\langle 1,1,1 \rangle$ and a $3-$point similarity embedding $Y$ such that $\lim_{t\searrow 0}{Y_t = Y}$ in Hausdorff distance but $\lim_{t \searrow 0} \rho_{t} > \rho$.

Let us put $\rho<\sqrt{\frac{2}{3}}$, and let $s>0$ and $Y_t =\rm \{A,B,C,D,E,F\} \in \bR^5$ be 
\begin{align*}
\mr A &= (-\rho,0,0,0,0), 
&\mr B &= \mr A+(t^2,0,st^2,\sqrt{2} t,0),  \\
\mr C &= \left( \frac{1}{2}\rho, \phantom{-}\frac{\sqrt{3}}{2}\rho, 0,0,0 \right), 
& \mr D &= \mr C+(0,-t^2,0,t,t), \\
\mr E&=\left( \frac{1}{2}\rho, -\frac{\sqrt{3}}{2}\rho, 0,0,0 \right),
&\mr F&=\mr E+(0,t^2,0,t,-t),
\end{align*}
Note that as $t \searrow 0$ points $\mr B,\mr D$ and $\mr F$ converge to $\mr A,\mr C$ and $\mr E$, respectively, so that $\lim_{t \searrow 0} Y_t = Y$ in Hausdorff distance and $\rho$ is indeed equal to the circumradius of $Y = \rm \{A,B,C\}$. 
On the other hand, the circumcenter and the circumradius of $Y_t$ are
\begin{align*}
\mr K_t &= \left(0,0, 
\frac{2-2\sqrt{2}+(\sqrt{6}-2)\rho}{2s} + \frac{1-\sqrt{2}+s^2}{2s} t^2   , 
\frac{2-\sqrt{3}\rho}{2} t +\frac{1}{2} t^3  ,0 \right) \\
\rho_t &= \sqrt{ \rho^2 + \left( \frac{2-2\sqrt{2}+(\sqrt{6}-2)\rho}{2s} \right)^2 + O(t^2) },
\end{align*}
so we have $\lim_{t \searrow 0} \rho_t > \rho$.
\end{ex}

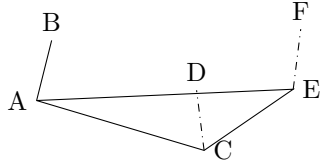
\begin{figure}
\centering
\begin{tikzpicture}
\tikzmath{
\radiusa=2.0; \radiusb=0.5; \thetaA=160;
\bx=0.2; \by=0.8;
\dx=-0.1; \dy=0.8;
\fx=0.1; \fy=0.8;
}
\draw ({\radiusa*cos(\thetaA)},{\radiusb*sin(\thetaA)}) node[anchor=east]{$\mr A$}
-- ({\radiusa*cos(\thetaA+120)},{\radiusb*sin(\thetaA+120)}) node[anchor=west]{$\mr C$}
-- ({\radiusa*cos(\thetaA+240)},{\radiusb*sin(\thetaA+240)}) node[anchor=west]{$\mr E$}
-- ({\radiusa*cos(\thetaA)},{\radiusb*sin(\thetaA)});

\draw ({\radiusa*cos(\thetaA)},{\radiusb*sin(\thetaA)}) 
-- ({\radiusa*cos(\thetaA)+\bx},{\radiusb*sin(\thetaA)+\by}) node[anchor=south]{$\mr B$};

\draw[dashdotted] ({\radiusa*cos(\thetaA+120)},{\radiusb*sin(\thetaA+120)})
-- ({\radiusa*cos(\thetaA+120)+\dx},{\radiusb*sin(\thetaA+120)+\dy}) node[anchor=south]{$\mr D$};

\draw[dashdotted]  ({\radiusa*cos(\thetaA+240)},{\radiusb*sin(\thetaA+240)})
--  ({\radiusa*cos(\thetaA+240)+\fx},{\radiusb*sin(\thetaA+240)+\fy}) node [anchor=south]{$\mr F$};
\end{tikzpicture}
\caption{The elements of $Y_t$ in Example \ref{ex:counter-(1,1,1)}. 
Four points $\mr A, \mr B, \mr C$ and $\mr D$ are on a common $3-$dimensional hyperplane on which $\mr E$ and $\mr F$ are not.}
\end{figure}

We will augment each $Y_t$ in Examples \ref{ex:counter-(4)}-\ref{ex:counter-(1,1,1)} to generate other counterexamples with various cardinalities and cluster types for Theorem \ref{thm:main}. 
We need a couple of additional lemmas for the augmentation process.

\begin{lem} \label{lem:add-1-pt}
Let $Y = \{y_1,\cdots,y_{k}\}$ be a similarity embedding in $\bR^{k-1} \subset \bR^{k}$ with circumradius $\rho_Y < 1$. 
For each $\varepsilon > 0$ there exists $y_{k}' \in \bR^{k}$ such that $\| y_{k} - y_{k}'\| < \varepsilon$, $|\rho_{Y\cup\{y_k'\}} - \rho_Y| < \varepsilon$ and $Y \cup \{y_k'\}$ is also a similarity embedding. 
\end{lem}

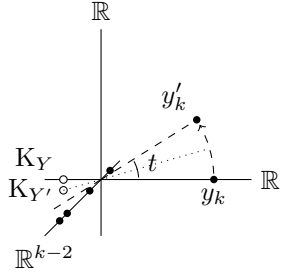
\begin{figure}
\centering
\begin{tikzpicture}
\tikzmath{
\radk = 1.5; \ang = 32;
\radik = 0.5; \radc = -0.5;
}
\draw (-0.75,0) -- (2,0) node[anchor=west]{$\bR$};
\draw (0.25,0.25) -- (-0.75,-0.75) node[anchor=north]{$\bR^{k-2}$};
\draw (0,-0.75) -- (0,2) node[anchor=south]{$\bR$};
\draw (\radk,0) node[fill,circle,inner sep=1pt]{};
\draw (\radk,0) node[anchor=north]{$y_k$};
\draw[dashed] (0,0) -- ({\radk*cos(\ang)},{\radk*sin(\ang)}) node[fill,circle,inner sep=1pt]{};
\draw ({\radk*cos(\ang)},{\radk*sin(\ang)}) node[anchor=south east]{$y_k'$};
\draw[dashed,->] ({\radk*cos(3)},{\radk*sin(3)}) arc(3:\ang-3:\radk);

\draw (\radc,0) node[fill,color=white,draw=black,circle,inner sep=1pt]{};
\draw (\radc,0) node[anchor=south east]{$\mr K_Y$};
\draw (\radc,{\radc*tan(\ang/2)}) node[fill,color=white,draw=black,circle,inner sep=1pt]{};
\draw[dotted] (\radc,{\radc*tan(\ang/2)}) node[anchor=east]{$\mr K_{Y'}$}
-- ({\radk*cos(\ang/2)},{\radk*sin(\ang/2)});
\draw[dashed] (0,0) -- ({1.5*\radc*cos(\ang)},{1.5*\radc*sin(\ang)});

\draw ({\radik*cos(3)},{\radik*sin(3)}) arc(3:\ang-3:\radik);
\draw ({\radik*cos(3)},{\radik*sin(3)-0.04}) node[anchor=south west]{$t$};
\draw (-0.55,-0.55) node[fill,circle,inner sep=1pt]{};
\draw (-0.45,-0.45) node[fill,circle,inner sep=1pt]{};
\draw (-0.15,-0.15) node[fill,circle,inner sep=1pt]{};
\draw (0.12,0.12) node[fill,circle,inner sep=1pt]{};
\end{tikzpicture}
\caption{The point $y_k'$ obtained by rotating $y_k$ and the circumcenters $\mr K_{Y'}$ and $\mr K_Y$.}
\label{fig:aug-lem-1}
\end{figure}

\begin{proof}
By applying a suitable rigid transform, we may assume that $y_1,\cdots,y_{k-1} \in \bR^{k-2} \times \{(0,0)\}$ and $y_{k} =(0,\cdots,0,a,0) \in \{(0,\cdots,0)\} \times \bR \times \{0\}$. 
Let $t$ be a real number and $y_k' = (0,\cdots,0,a \cos t, a \sin t) \in \bR^k$. 
If $\mr K_Y=(p,q,0) \in \bR^{k-2}\times \bR \times \{0\}$ is the circumcenter of $Y$, then $\|y_k - y_k'\| = 2|a| \sin \frac{t}{2}$ and the circumcenter $\mr K_{Y'}$ and the circumradius $\rho_{Y'}$ of $Y' = Y \cup \{y_k'\}$ are
\begin{align*}
\mr K_{Y'} & = \mr K_Y + \left( 0,\cdots,0,0, q  \tan \frac{t}{2} \right) \\
\rho_{Y'} & = \sqrt{ \rho_Y ^2  +  q^2 \tan^2 \frac{t}{2}}.
\end{align*}
(See Figure \ref{fig:aug-lem-1} for an illustration.) By letting $t$ sufficiently close to zero, both $\|y_k - y_k'\|$ and $|\rho_{Y'}-\rho_Y|$ can be made arbitrarily small. 
Moreover, when $t$ is sufficiently close to zero, the tri-similarity inequality \eqref{eq:strongly-acute} is satisfied for each triple of $Y'$. 
Indeed, if a triple in $Y'$ contains at most one of $y_k$ and $y_k'$, inequality \eqref{eq:strongly-acute} holds because both $Y$ and $\{y_1,\cdots,y_{k-1},y_k'\}$ are similarity embeddings. 
Otherwise, the triple is of the form $y_i,y_k,y_k'$ $(1\leq i < k)$ and inequality \eqref{eq:strongly-acute} holds because $\|y_k - y_k' \|$ can be made arbitrarily small and $\|y_i -y_k\| = \| y_i - y_k' \|$.
\end{proof}

\begin{lem} \label{lem:push-1-pt}
Let $Y_1$ and $Y_2 \supset Y_1$ be similarity embeddings in $\bR^d$ such that 
\begin{align*}
\rho_{Y_2} &< c_0 , &\sqrt{\rho_{Y_2}^2 - \rho_{Y_1}^2} &> c_1 , \\
\mathrm{diam\,} Y_2 &< c_2 \rho_{Y_2}  \leq 2\rho_{Y_2}, & c_0^{-1} + \sqrt{c_0^{-2}-1} &< c_3
\end{align*}
for some positive constants $c_0$, $c_1$, $c_2$, and $c_3$. 
If $y' \in \bR^{d+1}$ is the point such that the orthogonal projection of $y'$ to $\mathrm{aff\ } Y_2$ is equal to $\mr K_{Y_2}$ and $\|\mr K_{Y_2} - y' \|  = c_3\rho_{Y_2}$, then there exist $\delta_0,\delta_1 >0$ such that $Y_1' := Y_1 \cup \{y'\}$ and $Y_2' := Y_2 \cup \{y'\}$ satisfy the following conditions: 
\begin{itemize}
\item $\sqrt{\rho_{Y_2'}^2 - \rho_{Y_1'}^2}  > \delta_0$.
\item The dilated subsets $uY_1'$ and $uY_2'$ are similarity embeddings for all $u$ such that $0<u<\delta_1$.
\end{itemize}
\end{lem}

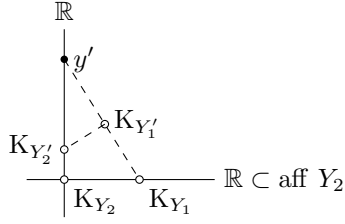
\begin{figure}
\centering
\begin{tikzpicture}
\tikzmath{
\gapot = 1; \yprime = 1.6; \ytp = 0.4;
\yopx = {(\gapot*\yprime*(\yprime-\ytp))/(\gapot*\gapot+\yprime*\yprime)};
\yopy = {(\yprime*(\gapot*\gapot+\yprime*\ytp))/(\gapot*\gapot+\yprime*\yprime)};
}
\draw (-0.5,0)--(2,0) node[anchor=west]{$\bR \subset\mr{aff\ } Y_2$};
\draw (0,-0.5)--(0,2) node[anchor=south]{$\bR$};
\draw (0,0) node[fill,color=white,draw=black,circle,inner sep=1pt]{};
\draw (0,0) node[anchor=north west]{$\mr{K}_{Y_2}$};
\draw (\gapot,0) node[fill,color=white,draw=black,circle,inner sep=1pt]{};
\draw (\gapot,0) node[anchor=north west]{$\mr{K}_{Y_1}$};

\draw (0,\yprime) node[anchor=west]{$y'$};
\draw[dashed] (0,\yprime) node[fill,circle,inner sep=1pt]{} -- (\gapot,0);

\draw (0,\ytp) node[fill,color=white,draw=black,circle,inner sep=1pt]{};
\draw (0,\ytp) node[anchor=east]{${\mr K}_{Y_2'}$};

\draw (\yopx,\yopy) node[fill,color=white,draw=black,circle,inner sep=1pt]{};
\draw[dashed] (\yopx,\yopy) node[anchor=west]{${\mr K}_{Y_1'}$} -- (0,\ytp);

\end{tikzpicture}
\caption{Added point $y'$ and the circumcenters $\mr K_{Y_1}, \mr K_{Y_2}, \mr K_{Y_1'},$ and $\mr K_{Y_2'}$.}
\label{fig:aug-lem-2}
\end{figure}

\begin{proof}
Since both vectors $y' - {\mr K}_{Y_2}$ and ${\mr K}_{Y_2}-{\mr K}_{Y_1}$ are orthogonal to $\mathrm{aff\ } Y_1$, so is $y' - {\mr K}_{Y_1} = (y' - {\mr K}_{Y_2} ) + ({\mr K}_{Y_2}-{\mr K}_{Y_1})$.
Then for $j=1,2$, by putting ${\mr K}_{Y_j'} = {\mr K}_{Y_j} + t_j (y'-{\mr K}_{Y_j})\ (t_j \in \bR)$ and using $\| y' - {\mr K}_{Y_j'} \|^2 = \| {\mr K}_{Y_j '} - y \|^2 = {\rho_{Y_j}^2 + \| {\mr K}_{Y_j} - {\mr K}_{Y_j'} \| ^2}$ for any $y \in Y_j$, we can find $t_j$ and that
\begin{align*}
\|y' -{\mr K}_{Y_1'} \| 
& = \frac{\|y' - {\mr K}_{Y_1}\|^2 + \rho_{Y_1}^2}{2 \| y' - {\mr K}_{Y_1} \|} \\
\|y' -{\mr K}_{Y_2'} \| 
& = \frac{\|y' - {\mr K}_{Y_2}\|^2 + \rho_{Y_2}^2}{2 \| y' - {\mr K}_{Y_2} \|} .
\end{align*}
Moreover, from
\begin{align*}
\|y' - {\mr K}_{Y_2}  \|^2 + \rho_{Y_2}^2 
& = \| y'-{\mr K}_{Y_1}\|^2 - \| {\mr K}_{Y_1} -{\mr K}_{Y_2}\|^2 + \rho_{Y_2}^2 \\
& = \| y' - {\mr K}_{Y_1} \|^2 - (\rho _{Y_2}^2 - \rho_{Y_1}^2 ) + \rho_{Y_2} ^2 \\
& = \| y' - {\mr K}_{Y_1} \|^2 + \rho_{Y_1}^2
\end{align*}
we obtain
\begin{equation*}
\frac{\| y' - {\mr K} _{Y_1'} \|}{\|y'-{\mr K_{Y_2'}}\| } =  \frac{\| y' - {\mr K}_{Y_2}\|}{\| y' - {\mr K}_{Y_1} \|}.
\end{equation*}
This implies the similarity of triangles $y'\mr K_{Y_1} \mr K_{Y_2}$ and $y' \mr K_{Y_2'} \mr K_{Y_1'}$, which in turn implies
\begin{equation*}
\frac{\| \mr K_{Y_1} - \mr K_{Y_2} \|}{\| \mr K_{Y_2'} - \mr K_{Y_1'} \|} 
= \frac{\| y' - \mr K_{Y_1} \|}{\| y' - \mr K_{Y_2'}\|}.
\end{equation*}
Then,
\begin{align*}
\sqrt{\rho_{Y_2'}^2 - \rho_{Y_1'}^2}
& = \| \mr K_{Y_2'} - \mr K_{Y_1'} \| \\
& = \frac{\| \mr K_{Y_1} - \mr K_{Y_2} \| \cdot \| y' - \mr K_{Y_2'}\|}{\| y' - \mr K_{Y_1} \|} \\ 
& = \frac{\sqrt{\rho_{Y_2}^2 - \rho_{Y_1}^2} \cdot \frac{1}{2} (c_3 +c_3^{-1})\rho_{Y_2}}{\sqrt{(\rho_{Y_2}^2 - \rho_{Y_1}^2) + (c_3 \rho_{Y_2})^2}}\\
& > \frac{c_1 (c_3 +c_3^{-1})c_1}{c_1 \sqrt{1+c_3^2}}.
\end{align*}
(See Figure \ref{fig:aug-lem-2} for an illustration.) 
Next, let us check the tri-similarity inequality in dilations $uY_2'$ of $Y_2'$ ($0<u\leq1$).
If a triple $(uy_1, uy_2, uy_3)$ in $uY_2'$ does not contain $uy'$, the tri-similarity inequality holds because $u\leq1$ and $Y_2$ is a similarity embedding by assumption. 
Otherwise, for a triple $(uy_1, uy_2, uy')$ we should find the range of $u$ such that both values
\begin{align*}
u^2 \| y_1 - y' \|^2 + u^2 \| y_2 -y' \|^2 - u^2 \| y_1 -y_2 \| ^2 - \frac{1}{2} u^4 \| y_1 -y' \|^2 \| y_2 -y' \|^2 \\
u^2 \| y_1 -y' \|^2 + u^2 \| y_1 - y_2 \| ^2 - u^2 \| y_2 -y' \|^2 -\frac{1}{2} u^4 \| y_1 -y'\|^2 \| y_1 -y_2 \|^2 
\end{align*}
become positive.
The upper one satisfies
\begin{align*}
&u^2 \cdot (1+c_3^2 ) \rho_{Y_2}^2 + u^2 \cdot (1+c_3^2 ) \rho_{Y_2}^2 - u^2 \| y_1 -y_2 \|^2 - \frac{1}{2} u^4 (1+c_3^2 )^2\rho_{Y_2}^4  \\
&= u^2 \left[ (2+2c_3^2 )\rho_{Y_2}^2 - \| y_1 -y_2 \|^2  - \frac{1}{2} u^2 (1+c_3^2 )^2 \rho_{Y_2}^4 \right] \\
&\geq u^2 \left[ (2+2c_3^2-c_2^2) \rho_{Y_2}^2 - \frac{1}{2} u^2 (1+c_1^2 )^2 \rho_{Y_2}^4 \right] \\
& = \frac{1}{2} u^2 \rho_{Y_2}^4  (1+c_3^2)^2 \left[ \frac{(4+4c_3^2-2c_2^2)}{\rho_{Y_2}^2(1+c_3^2)^2} - u^2\right]
\end{align*}
and is positive if $u < \frac{\sqrt{4+4c_3^2-2c_2^2} }{(1+c_3^2) c_1 }$. 
And the lower one satisfies
\begin{align*}
&u^2 \cdot (1+c_3^2)\rho_{Y_2}^2 + u^2 \| y_1 -y_2 \|^2 - u^2 \cdot (1+c_3^2)\rho_{Y_2}^2  - \frac{1}{2} u^4 \cdot (1+c_3^2)\rho_{Y_2}^2  \cdot \|y_1 -y_2 \|^2 \\
&= u^2 \| y_1 -y_2 \|^2  \left( 1 - \frac{1+c_3^2 }{2}  u^2 \rho_{Y_2}^2 \right)
\end{align*}
and is positive if $u < \frac{\sqrt{2}}{c_1\sqrt{1+c_3^2}}$.
\end{proof}

For two finite monotone decreasing sequences $r=\langle r_1,r_2,\cdots, r_m \rangle$ and $r'=\langle r_1',r_2',\cdots,r_{m}'\rangle $ of natural numbers, we will write $r \preceq r'$ if $r_i \leq r_i'$ for each $i=1,2,\cdots,m$.

\begin{lem} \label{lem:type-ext}
Let $(X_n)$ be a sequence of FPDMSs clustered in type $r=\langle r_1,r_2,\cdots,r_m \rangle$ such that $\lim_{n\rightarrow \infty}X_n = X$ in Gromov-Hausdorff distance but $\lim_{n \rightarrow \infty}|X_n| > |X|$. 
If we have either
\begin{enumerate}
\item $r'= \langle r_1',r_2',\cdots,r_{m}'\rangle \succeq r $, or
\item $r' = \langle r_1,r_2, \cdots,r_m,0,\cdots,0\rangle$, i.e. $r'$ is obtained by appending finitely many zeros to $r$,
\end{enumerate}
then there exists another sequence $(X_n')$ of FPDMSs clustered in type $r'$ such that $\lim_{n\rightarrow \infty}X_n' = X$ in Gromov-Hausdorff distance but $\lim_{n \rightarrow \infty}|X_n'| > |X|$.
\end{lem}

\begin{proof}
Using Corollary \ref{cor:X-vs-Y}, let $Y$ and $Y_n$ be similarity embeddings of $X$ and $X_n$, for each $n=1,2,\cdots$, respectively, such that $\lim_{n \rightarrow \infty} Y_n = Y$ in Hausdorff distance in $\bR^d$.
\begin{enumerate}
\item By applying Lemma \ref{lem:add-1-pt} iteratively, we can let $Y_n'$ another similarity embedding clustered in type $r'$ and satisfying $d_H(Y_n, Y_n') < \frac{1}{n}$ and $|\rho_{Y_n} - \rho_{Y_n'}| < \frac{1}{n}$. 
Then $(Y_n')$ converges to $Y$ in Hausdorff distance and $\lim_{n \rightarrow 0}(\rho_{Y}-\rho_{Y'}) = 0$.
Then the sequence $(X_n')$ recovered from $(Y_n')$ converges to $X$ by Corollary \ref{cor:X-vs-Y} and satisfy $\lim_{n \rightarrow \infty} (|X_n| - |X_n'|) = 0$.

\item For each $n=1,2,\cdots$, we can choose $Y_n^\ast \subset Y_n$ such that $\# Y_n^\ast = \#Y$ and $\lim_{n \rightarrow \infty} Y_n^\ast = Y$ in Hausdorff distance. 
Then we have $\lim_{n \rightarrow \infty}\rho_{Y_n^\ast} = \rho_Y$. 
By applying Lemma \ref{lem:push-1-pt} for each $n$ to $Y_n^\ast$ and $Y_n$, we obtain $\delta_0,\delta_1 > 0$, ${Y_n^{\ast}}'$ and $Y_n'$ such that $Y' := \lim_{n \rightarrow \infty}{Y_n^\ast}'$ converges in Hausdorff distance, $\sqrt{\rho_{Y_n'}^2 - \rho_{{Y_n^\ast}'}^2} > \delta_0 > 0$, and $uY_n'$ is a similarity embedding for all $0< u < \delta_1$. 
Thus, we obtain a sequence $(uY_n')$ of similarity embeddings clustered in type $\langle r_1,r_2,\cdots,r_m,0\rangle$ and converging to $uY'$. 
By using Corollary \ref{cor:X-vs-Y} after iterating the above process finitely many times, we can construct the desired sequence of FPDMSs.
\end{enumerate}
\end{proof}


\begin{proof}[Proof of Theorem \ref{thm:main}, ``only if'' part] Suitable counterexamples can be constructed, by first applying Lemma \ref{lem:type-ext} to ``amplify'' one of the similarity embeddings introduced in Section \ref{subsec:counterexamples} and then recovering FPDMSs by Corollary \ref{cor:X-vs-Y}. 
Let us refer to subsets mentioned in Examples \ref{ex:counter-(4)}, \ref{ex:counter-(3)}, \ref{ex:counter-(2,1)}, and \ref{ex:counter-(1,1,1)} as $Y_{\langle 4 \rangle}$, $Y_{\langle 3,0\rangle}$, $Y_{\langle 2,1\rangle}$, and $Y_{\langle 1,1,1 \rangle}$, respectively.

If $r$ is a finite monotone decreasing sequence of natural numbers such that $\|r\|_1 \geq 3$, then $r$ satisfies either $r \succeq \langle 3,0,\cdots,0 \rangle$ or $r \succeq \langle 2,1,0,\cdots,0 \rangle$ or $r \succeq \langle 1,1,1,0,\cdots,0 \rangle$. 
By assumption, the first case $r \succeq \langle 3,0,\cdots,0 \rangle$ implies either $\#X = 1, r \succeq \langle 4,0,\cdots,0 \rangle$ or $\#X = 2, r \succeq \langle 3,0,\cdots,0 \rangle$.
In each of these four sub-cases, we can construct counterexamples from $Y_{\langle 4 \rangle}$, $Y_{\langle 3,0\rangle}$, $Y_{\langle 2,1\rangle}$, and $Y_{\langle 1,1,1 \rangle}$, respectively.
\end{proof}

\bibliographystyle{abbrv}
\bibliography{references}

\begin{thebibliography}{10}

\bibitem{AsaoGomi2025}
Y.~Asao and K.~Gomi.
\newblock Geometric interpretation of magnitude, 2025.

\bibitem{BarceloCarbery2018}
J.~A. Barceló and A.~Carbery.
\newblock On the magnitudes of compact sets in euclidean spaces.
\newblock {\em American Journal of Mathematics}, 140:449--494, 4 2018.

\bibitem{Bunch+2020}
E.~Bunch, D.~Dickinson, J.~Kline, and G.~Fung.
\newblock Weighting vectors for machine learning: numerical harmonic analysis applied to boundary detection.
\newblock In {\em TDA \& Beyond}, 2020.

\bibitem{Devriendt2025}
K.~Devriendt.
\newblock The geometry of magnitude for finite metric spaces, 2025.

\bibitem{GimperleinGoffeng2021}
H.~Gimperlein and M.~Goffeng.
\newblock On the magnitude function of domains in euclidean space.
\newblock {\em American Journal of Mathematics}, 143:939--967, 2021.

\bibitem{Gimperlein+2022}
H.~Gimperlein, M.~Goffeng, and N.~Louca.
\newblock The magnitude and spectral geometry, 2024.

\bibitem{KaltonOstrovskii1999}
N.~J. Kalton and M.~I. Ostrovskii.
\newblock Distances between banach spaces.
\newblock {\em Forum Mathematicum}, 11, 1 1999.

\bibitem{Katsumasa+2025}
H.~Katsumasa, E.~Roff, and M.~Yoshinaga.
\newblock Is magnitude 'generically continuous' for finite metric spaces?, 2025.

\bibitem{Leinster2013}
T.~Leinster.
\newblock The magnitude of metric spaces.
\newblock {\em Doc. Math.}, 18:857--905, 2013.

\bibitem{LeinsterMeckes2016}
T.~Leinster and M.~W. Meckes.
\newblock Maximizing diversity in biology and beyond.
\newblock {\em Entropy}, 18, 2016.

\bibitem{LeinsterWillerton2013}
T.~Leinster and S.~Willerton.
\newblock On the asymptotic magnitude of subsets of euclidean space.
\newblock {\em Geometriae Dedicata}, 164:287--310, 2013.

\bibitem{Limbeck+2024}
K.~Limbeck, R.~Andreeva, R.~Sarkar, and B.~Rieck.
\newblock Metric space magnitude for evaluating the diversity of latent representations.
\newblock In A.~Globerson, L.~Mackey, D.~Belgrave, A.~Fan, U.~Paquet, J.~Tomczak, and C.~Zhang, editors, {\em Advances in Neural Information Processing Systems}, volume~37, pages 123911--123953. Curran Associates, Inc., 2024.

\bibitem{Meckes2013}
M.~W. Meckes.
\newblock Positive definite metric spaces.
\newblock {\em Positivity}, 17:733--757, 2013.

\bibitem{Meckes2015}
M.~W. Meckes.
\newblock Magnitude, diversity, capacities, and dimensions of metric spaces.
\newblock {\em Potential Analysis}, 42:549--572, 2015.

\bibitem{Memoli2008}
F.~Memoli.
\newblock Gromov-hausdorff distances in euclidean spaces.
\newblock In {\em 2008 IEEE Computer Society Conference on Computer Vision and Pattern Recognition Workshops}, pages 1--8, 2008.

\bibitem{Memoli2013}
F.~Mémoli.
\newblock The gromov-hausdorff distance: a brief tutorial on some of its quantitative aspects.
\newblock {\em Actes des rencontres du CIRM}, 3:89--96, 2013.

\end{thebibliography}

\appendix

\section{Supplement for Section \ref{subsec:counterexamples} } \label{sec:computation}

To assure that each Euclidean subset $Y_t$ of examples in Section \ref{subsec:counterexamples} is indeed a similarity embedding for all sufficiently small $t>0$, we verify the tri-similarity inequality \eqref{eq:strongly-acute} for every triple of the elements of $Y$. 
In the following computations, we check all the triangles whose vertices belong to $Y_t$, all three (instead of six, because of the symmetricity in inequality \eqref{eq:strongly-acute}) triples from each triangle, except those which follow from symmetricity. 

\begin{proof}[Supplement for Example \ref{ex:counter-(4)}]
The squared pairwise distances in $Y_t$ are
\begin{align*}
{\rm AB^2 = BC^2 = CA^2} &=  48 t^2 \\
{\rm DA^2 = DB^2 = DC^2  
= EA^2 = EB^2 = EC^2 }& = 25 t^2 \\
{\rm DE^2} & = 36t^2 - 36s^2 t^4.
\end{align*}
Then the tri-similarity inequality \eqref{eq:strongly-acute} is expressed as follows: 
for isosceles triangle $\rm DAB$ (and similarly for congruent triangles $\rm DBC$, $\rm DCA$, $\rm EAB$, $\rm EBC$ and $\rm ECA$)
\begin{align*} 
{\rm AD^2 + BD^2 - AB^2 - \frac{1}{2} AD^2 \cdot BD^2 }
& = 2t^2  + O(t^4) \\
{\rm AD^2 + AB^2 - BD^2 - \frac{1}{2} AD^2 \cdot AB^2} & = 48t^2  + O(t^4)
\end{align*}
and for isosceles triangle $\rm ADE$ (and similarly for congruent triangles $\rm BDE$ and $\rm CDE$)
\begin{align*}
{\rm AD^2 + AE^2 - DE^2 - \frac{1}{2} AD^2 \cdot AE^2 }
& = 14 t^2 + O(t^4) \\
{\rm AD^2 + DE^2 - AE^2 - \frac{1}{2} AD^2 \cdot DE^2 }
& = 36t^2 + O(t^4).
\end{align*}
We have skipped the regular triangle $\rm ABC$ because the computation is immediate.
These computations show that \eqref{eq:strongly-acute} indeed holds true for all triples, for all sufficiently small positive $t$. 
\end{proof}

\begin{proof}[Supplement for Example \ref{ex:counter-(3)}]
The squared pairwise distances in $Y_t$ are
\begin{align*}
{\rm AE^2=CE^2 }&= 4\rho^2  + (1+4\rho \cos s)t^2 \\
{\rm BE^2 = DE^2} &= 4\rho^2 + (1-4\rho \cos s)t^2 \\
{\rm AB^2=BC^2=CD^2=DA^2} &= \phantom{4\rho^2 + (1+\rho \cos s )}2t^2+2t^4 \\
{\rm AC^2 = BD^2}& = \phantom{4\rho^2 + (1+\rho \cos s )} 4t^2 -4t^4.
\end{align*}
Then the tri-similarity inequality \eqref{eq:strongly-acute} is expressed as follows:
for isosceles triangle $\rm ABC$ (and similarly for congruent triangles $\rm BCD$, $\rm CDA$ and $\rm DAB$) 
\begin{align*}
{\rm AB^2 + BC^2 - CA^2 - \frac{1}{2} AB^2\cdot BC^2 }
& = 6t^4 + O(t^6), \\
{\rm AB^2 + AC^2 - BC^2 - \frac{1}{2} AB^2 \cdot AC^2 } &= 4t^2 + O(t^4),
\end{align*}
for isosceles triangle $\rm EAC$
\begin{align*}
{\rm AE^2 +CE^2 - AC^2 -\frac{1}{2} AE^2 \cdot CE^2 }& = (8\rho^2 - 8\rho^4 )  + O(t^2) \\
{\rm AE^2 +AC^2 - CE^2  - \frac{1}{2} AE^2 \cdot AC^2 }&= (4-8\rho^2)t^2 + O(t^4),
\end{align*}
for isosceles triangle $\rm EBD$,
\begin{align*}
{\rm BE^2 +  DE^2 - BD^2 - \frac{1}{2} BE^2 \cdot DE^2 }&= (4\rho^2 - 8\rho^4)  + O(t^2) \\
{\rm BE^2 + BD^2 - DE^2 - \frac{1}{2} BE^2 \cdot BD^2 }& = (4 - 8\rho^2 )t^2  + O(t^4),
\end{align*}
and for triangle $\rm EAB$ (and similarly for congruent triangles $\rm EBC$, $\rm ECD$ and $\rm EDA$)
\begin{align*}
{\rm AE^2 + BE^2 - AB^2 - \frac{1}{2} AE^2 \cdot BE^2 }&= (8\rho^2 - 8\rho^4) + O(t^2 ) \\
{\rm AE^2 + AB^2 - BE^2 - \frac{1}{2} AE^2 \cdot AB^2 }&= ( 2 +8\rho \cos s - 4\rho^2  ) t^2 + O(t^4) \\
{\rm BE^2 + AB^2 - AE^2 - \frac{1}{2} BE^2 \cdot AB^2 }&= (2 - 8\rho \cos s - 4\rho^2 ) t^2 + O(t^4).
\end{align*}
These computations show that \eqref{eq:strongly-acute} indeed holds true for all triples, for all sufficiently small positive $t$. 
\end{proof}

\begin{proof}[Supplement for Example \ref{ex:counter-(2,1)}]
The squared pairwise distances in $Y_t$ are
\begin{align*}
{\rm AD^2 }&= 4\rho^2 , \\
{\rm BD^2 = CD^2 =AE^2 }& = 4\rho^2 + \phantom{{}-2\sqrt{3}} (4-4\rho)t^2 + (1+s^2)t^4, \\
{\rm BE^2 = CE^2 }& = 4\rho^2 + (8-2\sqrt{3}-8\rho)t^2 + 4t^4 \\
{\rm AB^2 = AC^2 = DE^2 }&= \phantom{4\rho^2 + ({}-2\sqrt{3}-8\rho)} 4t^2 + (1+s^2)t^4, \\
{\rm BC^2 }&=\phantom{4\rho^2 + ({}-2\sqrt{3}-8\rho)} 4t^2.
\end{align*}
Then the tri-similarity inequality \eqref{eq:strongly-acute} is expressed as follows:
for isosceles triangle $\rm ABC$
\begin{align*}
{\rm AB^2 + AC^2 - BC^2 - \frac{1}{2} AB^2 \cdot AC^2 }& = 4t^2 + O(t^4) \\
{\rm AB^2 + BC^2 - AC^2 - \frac{1}{2} AB^2 \cdot BC^2 }& = 4t^2 + O(t^4),
\end{align*}
for triangle $\rm ABD$ (and similarly in the congruent triangle $\rm ACD$ also)
\begin{align*}
{\rm AD^2 +BD^2 - AB^2  - \frac{1}{2} AD^2 \cdot BD^2 }
& = 8\rho^2 + O(t^2 ) \\
{\rm AD^2 + AB^2 - BD^2 - \frac{1}{2} AD^2 \cdot AB^2 }
& = (4\rho-8\rho^2 )t^2 + O(t^4) \\
{\rm AB^2 + BD^2 - AD^2 - \frac{1}{2} AB^2 \cdot BD^2 }
& = ( 8 -4\rho-8\rho^2) t^2 + O(t^4),
\end{align*}
for triangle $\rm ABE$ (and similarly in the congruent one $\rm ACE$ also)
\begin{align*}
{\rm AE^2 + BE^2 - AB^2 - \frac{1}{2} AE^2 \cdot BE^2 }
&= 8\rho^2 + O(t^2 ) \\
{\rm AB^2 +BE^2 - AE^2 - \frac{1}{2} AB^2 \cdot BE^2 }
& =( 8-2\sqrt{3} - 4\rho- 8\rho^2 )t^2 + O(t^4) \\
{\rm AE^2 + AB^2 - BE^2 - \frac{1}{2} AE^2 \cdot AB^2 }
& = ( 2\sqrt{3} + 4\rho - 8\rho^2 ) t^2 + O(t^4),
\end{align*}
for isosceles triangle $\rm BCD$
\begin{align*}
{\rm BD^2 +CD^2 - BC^2 - \frac{1}{2} BD^2 \cdot CD^2 }
& = 8\rho^2 + O(t^2 ) \\
{\rm BC^2 + CD^2 - BD^2 - \frac{1}{2} BC^2 \cdot CD^2 }
& = (4 - 8\rho^2) t^2 + O(t^2),
\end{align*}
for isosceles triangle $\rm BCE$
\begin{align*}
{\rm BE^2 + CE^2 - BC^2  - \frac{1}{2} BE^2 \cdot CE^2 }
&= 8\rho^2 + O(t^2 ) , \\
{\rm BC^2 + CE^2 - BE^2 - \frac{1}{2} BC^2 \cdot CE^2 }
& = (4 - 8\rho^2) t^2 + O(t^4 ),
\end{align*}
for triangle $\rm ADE$
\begin{align*}
{\rm AD^2 + AE^2 - AE^2 - \frac{1}{2} AD^2 \cdot AE^2 }
& = 8\rho^2 + O(t^2), \\
{\rm AD^2 + DE^2 - AE^2 - \frac{1}{2} AD^2 \cdot DE^2 }
& = (4\rho - 8\rho^2)t^2 + O(t^4), \\
{\rm AE^2 + DE^2 - AD^2 - \frac{1}{2} AE^2 \cdot DE^2 }
& = ( 8 -4\rho- 8\rho^2)t^2 + O(t^4),
\end{align*}
and for triangle $\rm BDE$ (and similarly for congruent triangle $\rm CDE$)
\begin{align*}
{\rm BD^2 + BE^2 - DE^2 - \frac{1}{2} BD^2 \cdot BE^2 }
& = 8\rho^2 + O(t^2 ) \\
{\rm BD^2 + DE^2 - BE^2 - \frac{1}{2} BD^2 \cdot DE^2 }
& = (2\sqrt{3} + 4\rho - 8\rho^2) t^2 + O(t^4) \\
{\rm BE^2 + DE^2 - BD^2 - \frac{1}{2} BE^2 \cdot DE^2 }
& = ( 8-2\sqrt{3} -4\rho- 8\rho^2) t^2 + O(t^4).
\end{align*}
These computations show that \eqref{eq:strongly-acute} indeed holds true for all triples, for all sufficiently small positive $t$. 
\end{proof}

\begin{proof}[Supplement for Example \ref{ex:counter-(1,1,1)}]
The squared pairwise distances in $Y_t$ are
\begin{align*}
{\rm AC^2 = CE^2 = EA^2 }
& = 3\rho^2, \\
{\rm AD^2 = AF^2 } 
& = 3\rho^2 + \phantom{\big[{}-2\sqrt{2}+1\big]}(2-\sqrt{3}\rho)t^2 + t^4 \\
{\rm BD^2 =BF^2 }
& = 3\rho^2 + \big[ 4-2\sqrt{2} - (3+\sqrt{3})\rho\big]t^2 + (2+s^2 )t^4 \\
{\rm CB^2  =EB^2 }
& = 3\rho^2 + \phantom{\big[{}-2\sqrt{2}+\sqrt{2}\big]}(2-3\rho)t^2 + (1+s^2 )t^4 \\
{\rm CF^2  = ED^2 }
& = 3\rho^2 + \phantom{\big[{}-2+\sqrt{2}\big]}(2 -2\sqrt{3}\rho)t^2 + t^4 \\
{\rm DF^2 }
&= 3\rho^2 + \phantom{\big[{}-2+\sqrt{2}\big]}(4-4\sqrt{3}\rho ) t^2 + 4t^4, \\
{\rm AB^2 }
&= \phantom{3\rho^2 + \big[4-\sqrt{2}-(3+\sqrt{3})\rho\big]} 2t^2 + (1+s^2 ) t^4 , \\
{\rm CD^2 = EF^2 }
&=\phantom{3\rho^2 + \big[4-\sqrt{2}-(3+\sqrt{3})\rho\big]}2t^2 +t^4 , 
\end{align*}
Then the tri-similarity inequality \eqref{eq:strongly-acute} is expressed as follows: 
for triangle $\rm BDF$
\begin{align*}
{\rm BD^2 +DF^2 - BF^2 - \frac{1}{2} BD^2 \cdot DF^2 }
& =  \left( 6\rho^2 - \frac{9}{2} \rho^4 \right) + O(t^2 ) \\
{\rm DF^2 +BF^2 - BD^2 - \frac{1}{2} DF^2 \cdot BF^2 }
& =  \left( 6\rho^2 - \frac{9}{2} \rho^4 \right) + O(t^2 ) \\
{\rm BF^2 +BD^2 - DF^2 - \frac{1}{2} BF^2 \cdot BD^2 }
& =  \left( 6\rho^2 - \frac{9}{2} \rho^4 \right) + O(t^2 ),
\end{align*}
in quadrilateral $\rm ABCD$ (and similarly for congruent quadrilateral $\rm ABEF$), for triangle $\rm ABC$
\begin{align*}
{\rm AC^2 + BC^2 - AB^2 - \frac{1}{2} AC^2 \cdot BC^2 }
&= \left(6\rho^2 - \frac{9}{2}\rho^4 \right) + O(t^2) \\
{\rm AB^2 + BC^2 - AC^2 - \frac{1}{2} AB^2 \cdot BC^2 }
& = (4-3\rho-3\rho^2)t^2 + O(t^4) ,\\
{\rm AC^2 + AB^2 - BC^2 - \frac{1}{2} AC^2 \cdot AB^2 }
& = (3\rho - 3\rho^2 ) t^2  + O(t^4),
\end{align*}
for triangle $\rm ACD$
\begin{align*}
{\rm AC^2 + AD^2 - CD^2 - \frac{1}{2} AC^2 \cdot AD^2 }
& = \left(6\rho^2 - \frac{9}{2} \rho^4 \right) + O(t^2 ) \\
{\rm AC^2 + CD^2 - AD^2 - \frac{1}{2} AC^2 \cdot CD^2 }
& = (\sqrt{3}\rho - 3\rho^2 )t^2 + O(t^4) \\
{\rm AD^2 + CD^2 - AC^2 - \frac{1}{2} AD^2 \cdot CD^2 }
& = (4-\sqrt{3} \rho - 3\rho^2 ) t^2 + O(t^4),
\end{align*}
for triangle $\rm BCD$
\begin{align*}
{\rm BC^2 + BD^2 - CD^2 - \frac{1}{2} BC^2 \cdot BD^2 }
& = \left(6\rho^2 - \frac{9}{2} \rho^4 \right) + O(t^2) \\
{\rm BC^2 + CD^2 - BD^2 - \frac{1}{2} BC^2 \cdot CD^2 }
& = (2\sqrt{2} +\sqrt{3} \rho - 3\rho^2 )t^2 + O(t^4) \\
{\rm BD^2 + CD^2 - BC^2 - \frac{1}{2} BD^2 \cdot CD^2 }
& = (4-2\sqrt{2}  -\sqrt{3}\rho - 3\rho^2 )t^2 + O(t^4),
\end{align*}
for triangle $\rm ABD$
\begin{align*}
{\rm AD^2 + BD^2 - AB^2 - \frac{1}{2} AD^2 \cdot BD^2 }
& = \left(6\rho^2 - \frac{9}{2} \rho^4 \right) + O(t^2 ) \\
{\rm AD^2 + AB^2 - BD^2 - \frac{1}{2} AD^2 \cdot AB^2 }
& = (2\sqrt{2} +3\rho - 3\rho^2 )t^2 + O(t^4) \\
{\rm BD^2 + AB^2 - AD^2 - \frac{1}{2} BD^2 \cdot AB^2 }
& = (4-2\sqrt{2} -3\rho -3\rho^2 )t^2 + O(t^4 ),
\end{align*}
in quadrilateral $\rm CDEF$, for triangle $\rm CDE$ (and similarly for congruent triangle $\rm CEF$)
\begin{align*}
{\rm CE^2 + DE^2 - CD^2 - \frac{1}{2} CE^2 \cdot DE^2 }
& = \left(6\rho^2 - \frac{9}{2} \rho^4\right) + O(t^2) \\
{\rm CD^2 + CE^2 - DE^2 - \frac{1}{2} CD^2 \cdot CE^2 }
& = (2\sqrt{3}\rho - 3\rho^2 ) t^2 + O(t^4 ) \\
{\rm CD^2 + DE^2 - CE^2 - \frac{1}{2} CD^2 \cdot DE^2 }
& = (4-2\sqrt{3}\rho - 3\rho^2 ) t^2 + O(t^4 ),
\end{align*}
and for triangle $\rm CDF$ (and similarly for congruent triangle $\rm DEF$)
\begin{align*}
{\rm CF^2 + DF^2 - CD^2 - \frac{1}{2} CF^2 \cdot DF^2 }
& = \left(6\rho^2 - \frac{9}{2}\rho^4\right) + O(t^2) \\
{\rm CF^2 + CD^2 - DF^2 - \frac{1}{2} CF^2 \cdot CD^2 }
& = (2\sqrt{3}\rho - 3\rho^2 )t^2 + O(t^4) \\
{\rm DF^2 + CD^2 - CF^2 - \frac{1}{2} DF^2 \cdot CF^2 }
& = (4 -2\sqrt{3}\rho - 3\rho^2 ) t^2 + O(t^4).
\end{align*}
We have skipped the regular triangle $\rm ACE$ because the computation is immediate. 
These computations show that \eqref{eq:strongly-acute} indeed holds true for all triples, for all sufficiently small positive $t$. 
\end{proof}

\end{document}